\documentclass{sinst}

\usepackage{amsbsy,amsfonts,amsmath,amssymb,amsthm}
\usepackage{array}
\usepackage{bm}
\usepackage{color}
\usepackage{enumerate,enumitem}
\usepackage{mathrsfs}
\usepackage{latexsym}
\usepackage[colorlinks=true]{hyperref}
\usepackage{mathtools}
\mathtoolsset{showonlyrefs, showmanualtags}
\hypersetup{urlcolor=black, citecolor=black, linkcolor=black}

\definecolor{wineRed}{rgb}{0.7,0,0.3}
\definecolor{grandBleu}{rgb}{0,0,0.8}
\definecolor{darkGreen}{rgb}{0,0.4,0}
\definecolor{blueViolet}{rgb}{0.4,0,1.0}
\definecolor{bloodOrange}{rgb}{0.85,0.05,0}
\definecolor{mycolor}{rgb}{0.8,0,0.2}

\usepackage[textsize=small]{todonotes}
\usepackage{cite}

\usepackage{comment}

\allowdisplaybreaks[4]

\DeclareMathAlphabet{\mathpzc}{OT1}{pzc}{m}{it}

\usepackage[textsize=small]{todonotes}
\numberwithin{equation}{section}

\theoremstyle{plain}
\newtheorem{main theorem}{Main Theorem} 
\newtheorem{lemma}{Lemma}[section]
\newtheorem{key lemma}{Key-Lemma}

\theoremstyle{definition}
\newtheorem{definition}{Definition}
\newtheorem{remark}{Remark}

\newtheorem{example}{Example}

\def\N{\mathbb{N}}
\def\R{\mathbb{R}}

\def\ds{\displaystyle}

\DeclareMathOperator{\diver}{div}

\def\ae{\text{ a.e. }}
\def\aein{\text{ a.e. in }}
\def\aeon{\text{ a.e. on }}
\def\as{\text{ as }}
\def\In{\text{ in }}

\def\If{\text{ if }}

\def\forae{\text{ for a.e. }}
\def\II{\text{(I\hspace{-1pt}I)}~}

\def\weakly{\text{ weakly}}
\def\weaklystar{\text{ weakly-}\ast}

\def\otherwise{\text{ otherwise}}

\def\loc{\text{loc}}

\def\ds{\displaystyle}
\def\norm#1{\left|#1\right|}

\def\autom#1#2{\left\{#1\,\middle|\,#2\right\}}

\begin{document}
\label{page:t}
\thispagestyle{plain}

\title{\vspace{-3ex}
LARGE-TIME BEHAVIOR OF PSEUDO-PARABOLIC EQUATIONS ASSOCIATED WITH \\ GENERALIZED TOTAL VARIATION ENERGIES \footnotemark[1]
\vspace{-1.5ex}
}
\author{Daisuke Kubota
}
\affiliation{Division of Mathematics and Informatics, Department of Mathematics and Informatics, \\ Graduate School of Science and Engineering, Chiba University, \\ 1--33, Yayoi-cho, Inage-ku, 263--8522, Chiba, Japan}
\email{26wd0102@student.gs.chiba-u.jp
\vspace{-1.5ex}
}
\vspace{-3ex}
\sauthor{Daiki Mizuno
}
\saffiliation{Waseda Junior and Senior High School, \\ 62 Babashita-cho, Shinjuku-ku, 162--8654, Tokyo, Japan}
\semail{dmizuno.m@gmail.com
\vspace{-1.5ex}
}
\vspace{-3ex}
\tauthor{Ken Shirakawa
}
\taffiliation{Department of Mathematics, Faculty of Education, Chiba University \\ 1--33 Yayoi-cho, Inage-ku, 263--8522, Chiba, Japan}
\temail{sirakawa@faculty.chiba-u.jp
\vspace{-1.5ex}
}
\vspace{-3ex}
\fauthor{Naotaka Ukai
}
\faffiliation{Division of Mathematics and Informatics, Department of Mathematics and Informatics, \\ Graduate School of Science and Engineering, Chiba University, \\ 1--33, Yayoi-cho, Inage-ku, 263--8522, Chiba, Japan}
\femail{24wd0101@student.gs.chiba-u.jp}
\footcomment{
$^*$\,This work was supported by JST SPRING, Grant Number JPMJSP2109, and JSPS KAKENHI Grant Number JP26K06860. 
}
\maketitle
\noindent
{\bf
Abstract.} 
In this paper, we study the well-posedness and large-time behavior of a pseudo-parabolic problem associated with a generalized total variation energy in the $BV$-framework. A main mathematical issue arises from the mismatch between the Sobolev regularity of solutions at finite times and the $BV$-structure of the corresponding steady-state problem. We prove the existence and uniqueness of solutions and investigate the relationship between their large-time behavior and solutions to the steady-state problem. Moreover, our results include the convergence of the solution trajectory to the unique steady-state solution under a typical setting arising in image processing.
\pagebreak

\section*{Introduction}
Let $N \in \N$ be a fixed space dimension, and let $\Omega \subset \R^N$ be a bounded domain with boundary $\Gamma := \partial \Omega$. In particular, when $N > 1$, we assume that $\Gamma$ has at least Lipschitz regularity and denote by $n_\Gamma$ the unit outer normal to $\Gamma$.

For a fixed parameter $\varepsilon \in[0,1)$, we first consider the following nonlinear boundary value problem, denoted by $(\mathrm{S})_\varepsilon$:
\begin{gather}\label{E-L_Eq}
    (\mathrm{S})_\varepsilon~~~~
    \begin{cases}
        \ds
        -\diver\left(
            \frac{Dw}{\sqrt{\varepsilon^2 + |Dw|^2}}
        \right)
        + g(w) = f_\infty(x),
        ~\text{a.e.  $ x \in \Omega$,}
        \\[3ex]
        \bigl(
            \frac{Dw}{\sqrt{\varepsilon^2 + |Dw|^2}}
        \bigr) \cdot n_\Gamma = 0,
        ~\text{a.e. on $\Gamma$.}
    \end{cases}
\end{gather}
The problem $(\mathrm{S})_\varepsilon$ is associated with the Euler--Lagrange equation for the following energy functional $  \mathcal{F}_\varepsilon : L^2(\Omega) \longrightarrow (-\infty,\infty] $, defined by
\begin{gather}
    \mathcal{F}_\varepsilon(w)
    :=
    \begin{cases}
        \multicolumn{2}{l}{\ds
        \int_\Omega \sqrt{\varepsilon^2 + |Dw|^2}
        + \int_\Omega G(w)\,dx
        - \int_\Omega f_\infty w\,dx,}
        \\
        & \text{if $w \in D(\mathcal{F}_\varepsilon) := BV(\Omega) \cap L^2(\Omega)$,}
        \\[1ex]
        \infty,
        & \text{otherwise.}
    \end{cases}
\end{gather}
In this context, $ f_\infty : \Omega \longrightarrow \R $ is a given function, and $ g : \R \longrightarrow \R $ is a given perturbation with a potential function $ G : \R \longrightarrow \R $. 
For $w \in BV(\Omega)$, the quantity $ \int_\Omega \sqrt{\varepsilon^2 + |Dw|^2} $ denotes the generalized total variation defined by (cf. \cite{MR0775682}):
\begin{gather}\label{TV_eps}
    \int_\Omega \sqrt{\varepsilon^2 +|Dw|^2} 
    := \sup \left\{ 
    \begin{array}{l|l}
        \ds \int_\Omega \bigl( \varepsilon \varphi_0 +w \diver \bm{\varphi} \bigr) \, dx 
        &
        \parbox{4cm}{
            $ [\varphi_0, \bm{\varphi}] \in [C_\mathrm{c}^1(\Omega)]^{1 +N} $, 
            \\[1ex] 
            $ \bigl| [\varphi_0, \bm{\varphi}] \bigr|_{\R^{1 +N}} \leq 1 $ on $ \Omega $
        }
    \end{array} \right\}.
\end{gather}
The principal characteristic of the above energy $ \mathcal{F}_\varepsilon $ lies in the linear growth part of the generalized total variation. Such linear growth energy plays an important role in various variational problems arising in materials science and image processing (cf. \cite{
    MR2033382,
    MR2368971,
    MR1883415,
    MR2139202,
    MR2178065,
    MR3363401,
    rudin1994total,
    MR2028838,
    MR2028858,
    MR2170510,
    MR2746654,
    MR1712447,
    MR1847840,
    MR1851862,
    MR2571495,
    MR2232843
}). Accordingly, the Euler--Lagrange equation (S)$_\varepsilon$ has also attracted considerable attention as a problem characterizing stationary points of the energy. Moreover, (S)$_\varepsilon$ arises as the steady-state problem in various time-evolution models.

For instance, in image processing, $g$ is often taken to be positive and linear, while $f_\infty$ represents observed image data contaminated by noise. In this setting, the energy $\mathcal{F}_\varepsilon$ can be interpreted as a total variation type energy for image restoration. Its minimization problem and the associated time-evolution problems have been extensively studied as mathematical models for image denoising (cf. 
\cite{MR2033382,
    MR2368971,
    MR1883415,
    MR2139202,
    MR2178065,
    MR3363401,
    rudin1994total,
    MR2028838,
    MR2028858,
    MR2170510
}). On the other hand, in materials science, Visintin (cf. \cite{MR1423808}) proposed a diffusive interface model of phase-transition which allows discontinuities in the phase-field and thereby enables the description of sharp-interface-type structures. A similar idea was incorporated into a mathematical model for grain-boundary motion by Kobayashi et al. (cf. \cite{MR1752970,MR1794359}). In these models, problems corresponding to (S)$_\varepsilon$ are regarded as steady-state problems characterizing possible limiting states of the time evolution.

Among previous studies on this class of problems, the case $\varepsilon=0$ has been particularly well investigated (cf. \cite{
    MR2033382,
    MR2368971,
    MR1883415,
    MR2139202,
    MR2178065,
    MR3363401,
    rudin1994total,
    MR2028838,
    MR2028858,
    MR2170510,
    MR2746654,
    MR1712447,
    MR1847840,
    MR1851862,
    MR2232843
}). In this case, the principal nonlinear diffusion term $-\diver \bigl(\frac{Dw}{|Dw|}\bigr) $ corresponds to a very singular case. Owing to this singular diffusion structure and the linear-growth nature of the energy, the Sobolev space $H^1(\Omega)$ is not sufficient for treating the steady-state problem $(S)_0$ in general. In fact, spatially discontinuous $BV$-solutions are known to exist, including explicit examples (cf. 
\cite{
    MR2033382,
    MR2368971,
    MR1883415,
    MR2139202,
    MR2178065,
    MR1847840,
    MR1851862,
    MR4352617,
    MR2232843
}).

On the other hand, when $\varepsilon \in (0, 1)$, the corresponding diffusion $ \textstyle -\diver \bigl( \frac{Dw}{\sqrt{\varepsilon^2 +|Dw|^2}} \bigr) $ no longer falls into the very singular case due to the regularizing effect of $ \varepsilon \in (0, 1) $. Nevertheless, the linear-growth property of the generalized total variation remains unchanged. Therefore, even for $\varepsilon \in (0, 1)$, the natural space for treating solutions remains $BV(\Omega)$, and examples of discontinuous $BV$-solutions with a jump part have in fact been reported for the steady-state problem (S)$_\varepsilon$ (cf. \cite{
    MR2571495
}). 
Thus, for both $\varepsilon=0$ and $\varepsilon>0$, the essential feature of (S)$_\varepsilon$ lies in its nonsmooth variational structure on $BV(\Omega)$, beyond the usual Sobolev-space framework.

In this paper, we focus on the following pseudo-parabolic initial-boundary problem, denoted by (P)$_\varepsilon$:
\begin{gather}\label{(P)_eps}
    (\mathrm{P})_\varepsilon~~~~
    \begin{cases}
        \ds
        \alpha(t, x) \partial_t u -\diver\biggl(
            \frac{\nabla u}{\sqrt{\varepsilon^2 + |\nabla u|^2}}
            +\beta(t, x) \nabla \partial_t u
            \biggr)
        + g(u) = f(t, x),
        \\[2ex]
        \qquad \text{a.e. in $ (t, x) \in Q := (0, \infty) \times \Omega$,}
        \\[2ex]
        \bigl(
            \frac{\nabla u}{\sqrt{\varepsilon^2 + |\nabla u|^2}} +\beta(t, x) \nabla \partial_t u
        \bigr) \cdot n_\Gamma = 0,
        ~\text{a.e. $ (t, x) \in \Sigma := (0, \infty) \times \Gamma$,}
        \\[2ex]
        u(0, x) = u_0(x), ~\mbox{a.e. $ x \in \Omega $,}
    \end{cases}
\end{gather}
as a time-evolution problem having $(\mathrm{S})_\varepsilon$ as its
corresponding steady-state problem. 
Here, $f:Q\longrightarrow\R$ is a forcing term,  and $u_0:\Omega\longrightarrow\R$ is the initial datum. Additionally, $\alpha:Q\longrightarrow(0,\infty)$ is a given weight function, and $\beta:Q\longrightarrow(0,\infty)$ is a given Lipschitz weight function satisfying $ \partial_t \beta \leq 0 $ in $ Q $.

An important feature of the problem (P)$_\varepsilon$ is that the coefficients $\alpha$ and $\beta$ are allowed to depend on both time and space. Here, $\alpha(t,x)$ represents a local weight for the time evolution, whereas $\beta(t,x)$ determines the strength of the pseudo-parabolic term $-\diver\bigl(\beta(t,x)\nabla\partial_tu\bigr)$. The pseudo-parabolic term provides additional spatial regularity for the nonlinear diffusion associated with the generalized total variation on $ BV(\Omega) $, given in \eqref{TV_eps}. 
By allowing $\beta$ to depend on both time and space, the framework can describe not only spatially non-uniform regularization effects but also situations in which their strength varies and can be adjusted along the time evolution. 

Note that the well-posedness of (P)$_\varepsilon$ does not follow directly from the existing results for $\alpha$ and $\beta$ constant in time (cf. \cite{aiki2023class}). In particular, the time-dependence of $\alpha$ and $\beta$ affects both the energy estimates and the uniqueness argument. It is therefore necessary to establish the existence and uniqueness of solutions under the present general setting.

Now, we turn to the main issue of this paper, namely, the large-time behavior of (P)$_\varepsilon$. 
When we assume:
\begin{itemize}
    \item[$\sharp$0)]$f_\infty$ corresponds to some large-time limit of the forcing $f$ in an appropriate sense, 
\end{itemize}
\noindent
the Euler--Lagrange equation (S)$_\varepsilon$ naturally arises as the steady-state problem associated with the large-time behavior of solutions to (P)$_\varepsilon$. 
Then, on every finite time interval, the regularizing effect of the pseudo-parabolic term ensures that the solution to (P)$_\varepsilon$ has the regularity of the usual Sobolev space $H^1(\Omega)$. In contrast, as mentioned above, the natural solution space for the corresponding steady-state problem (S)$_\varepsilon$ is $BV(\Omega)$, and some steady-state solutions may have spatial discontinuities and thus fail to belong to $H^1(\Omega)$.

Therefore, particularly from the viewpoint of image processing, the pseudo-parabolic dynamics governed by (P)$_\varepsilon$ has an attractive feature for applications. Namely, it maintains spatial regularity at every finite time, while allowing its steady states to have discontinuities corresponding to image contours and edges.

However, this feature also gives rise to an essential difficulty in the mathematical analysis, due to the intrinsic mismatch between the solution space of the time-evolution problem at finite times, and the natural solution space of the steady-state problem describing its large-time limit. Indeed, the energy estimates yield the strong dissipation property $\partial_tu(\cdot+s)\to0$ in $L^2(0,1;L^2(\Omega))$ as $s\to\infty$. Nevertheless, a time-uniform $H^1(\Omega)$-estimate for the state $u(\cdot+s)$ is not available in general. Consequently, the large-time limit cannot be identified as a solution to (S)$_\varepsilon$ by the standard Minty's trick based on monotonicity of nonlinear diffusions (cf. \cite{MR348562,MR2582280}).

The principal aim of this paper is to clarify the relationship between the large-time behavior of solutions to (P)$_\varepsilon$ and the steady-state problem (S)$_\varepsilon$ under this mismatch between the $H^1(\Omega)$-regularity at finite times and the $BV(\Omega)$-structure in the large-time limit. 

Our main results can be summarized as follows.
\\[-3.5ex]
\begin{description}
\item[{\boldmath\underline{Existence and uniqueness of global-in-time solution.}}]~~~~
    \vspace{-1ex}
\item[~~Main Theorem 1:]
    We establish the global existence and uniqueness of solutions to (P)$_\varepsilon$ with time-space-dependent coefficients $\alpha$ and $\beta$. 
    \vspace{-1ex}
\item[{\boldmath\underline{Large-time behavior of solution.}}]
    We establish the following three levels of large-time behavior for the solution $u$ to the problem (P)$_\varepsilon$.
    \vspace{-1ex}
\item[~~Main Theorem 2:]
    Under the basic assumptions including $\sharp$0) and $ \partial_t \beta \leq 0 $ in $ Q $, at least one $\omega$-limit point of $u$ as $t \to \infty$ is shown to be a solution to the steady-state problem (S)$_\varepsilon$.
    \vspace{-1ex}
\item[~~Main Theorem 3:]
Under the additional assumption that $\beta$ decays exponentially in time, every $\omega$-limit point of $u$ is shown to be a solution to the steady-state problem (S)$_\varepsilon$.
    \vspace{-1ex}
\item[~~Main Theorem 4:]
Under additional assumptions ensuring the uniqueness of the steady-state solution, 
        the whole trajectory $ \{ u(t) \}_{t \geq 0} \subset H^1(\Omega) $ of solution $ u $ is shown to converge to the unique steady-state solution in $ BV(\Omega) $. In particular, this result covers the case where $\alpha$ and $\beta$ are positive constants and is therefore applicable to basic settings arising in image processing.
\end{description}

\section{Preliminaries}
We begin by prescribing the notations used throughout this paper. 
\bigskip

\noindent
\underline{\textbf{\textit{Notations in real analysis.}}}
We define:
\begin{align*}
    & r \vee s := \max \{ r, s \} ~ \mbox{ and } ~ r \wedge s := \min \{r, s\}, \mbox{ for all $ r, s \in [-\infty, \infty] $,}
\end{align*}
and especially, we write:
\begin{align*}
    & [r]^+ := r \vee 0 ~ \mbox{ and } ~ [r]^- := -(r \wedge 0), \mbox{ for all $ r \in [-\infty, \infty] $.}
\end{align*}
Additionally, for any $M > 0$, let $\mathcal T_M : \mathbb R \longrightarrow [ -M, M ]$ be the truncation operator, defined as
\begin{gather*}
  \mathcal T_M : r \in \mathbb R \mapsto \bigl( r \lor ( -M ) \bigr) \land M \in [ -M, M ].
\end{gather*}

Throughout this paper, we frequently use the following elementary fact from real analysis.
\noindent
    \begin{description}
        \item[{\boldmath\textbf{Fact\,($\ast$):}}] Let $a,b \in \mathbb R$ and $\{a_n\}_{n \in \mathbb N},\{b_n\}_{n\in\mathbb N} \subset \mathbb R$ be such that
\begin{gather*}
 \liminf_{n \to \infty} a_n \geq a,\ \liminf_{n\to \infty}b_n \geq b, \text{ and } \limsup_{n \to \infty}(a_n + b_n) \leq a + b.
\end{gather*}
    Then, it holds that $a_n \to a$ and $b_n \to b$ as $n \to \infty$.
\end{description}

Let $ d \in \N $ be a fixed dimension. We denote by $ |y| $ and $ y \cdot z $ the Euclidean norm of $ y \in \mathbb{R}^d $ and the scalar product of $ y, z \in \R^d $, respectively, i.e., 
\begin{equation*}
\begin{array}{c}
| y | := \sqrt{y_1^2 +\cdots +y_d^2} \mbox{ \ and \ } y \cdot z  := y_1 z_1 +\cdots +y_d z_d, 
\\[1ex]
\mbox{ for all $ y = [y_1, \ldots, y_d], ~ z = [z_1, \ldots, z_d] \in \mathbb{R}^d $.}
\end{array}
\end{equation*}
We denote by $\mathcal{L}^{d}$ the $ d $-dimensional Lebesgue measure, and we denote by $ \mathcal{H}^{d} $ the $ d $-dimensional Hausdorff measure.  In particular, the measure theoretical phrases, such as ``a.e.'', ``$dt$'', and ``$dx$'', and so on, are all with respect to the Lebesgue measure in each corresponding dimension. Also on a Lipschitz-surface $ S $, the phrase ``a.e.'' is with respect to the Hausdorff measure in each corresponding Hausdorff dimension. In particular, if $S$ is $C^1$-surface, then we simply denote by $dS$ the area-element of the integration on $S$.

For a Borel set $ E \subset \R^d $, we denote by $ \chi_E : \R^d \longrightarrow \{0, 1\} $ the characteristic function of $ E $. Additionally, for a distribution $ \zeta $ on an open set in $ \R^d $ and any $i \in \{ 1,\dots,d \}$, let $ \partial_i \zeta$ be the distributional differential with respect to $i$-th variable of $\zeta$. As well as we consider, the differential operators, such as $\nabla,\ \diver, \ \nabla^2$, and so on, are considered in distributional senses.
\bigskip

\noindent
\underline{\textbf{\textit{Abstract notations. (cf. \cite[Chapter II]{MR0348562})}}}
For an abstract Banach space $ X $, we denote by $ |\cdot|_{X} $ the norm of $ X $, and denote by $ \langle \cdot, \cdot \rangle_X $ the duality pairing between $ X $ and its dual $ X^* $. In particular, when $ X $ is a Hilbert space, we denote by $ (\cdot,\cdot)_{X} $ the inner product of $ X $.

For Banach spaces $ X_1, \dots, X_d $ with $ 1 < d \in \N $, let $ X_1 \times \dots \times X_d $ be the product Banach space endowed with the norm $ |\cdot|_{X_1 \times \cdots \times X_d} := |\cdot|_{X_1} + \cdots +|\cdot|_{X_d} $. However, when all $ X_1, \dots, X_d $ are Hilbert spaces, $ X_1 \times \dots \times X_d $ denotes the product Hilbert space endowed with the inner product $ (\cdot, \cdot)_{X_1 \times \cdots \times X_d} := (\cdot, \cdot)_{X_1} + \cdots +(\cdot, \cdot)_{X_d} $ and the norm $ |\cdot|_{X_1 \times \cdots \times X_d} := \bigl( |\cdot|_{X_1}^2 + \cdots +|\cdot|_{X_d}^2 \bigr)^{\frac{1}{2}} $. In particular, when all $ X_1, \dots,  X_d $ coincide with a Banach space $ Y $, the product space $X_1 \times \dots \times X_d$ is simply denoted by $[Y]^d$.
\bigskip

\noindent
\underline{\textbf{\textit{Basic notations.}}}
In this paper, we set the time-space cylindrical domain and its lateral boundary as follows:
\begin{gather*}
    Q_T \coloneqq (0,T) \times \Omega \text{ and }\Sigma_T \coloneqq (0,T) \times \Gamma, ~\mbox{for any $ T \in (0, \infty) $},
    \\
    \mbox{and }~~ 
  Q \coloneqq (0,\infty) \times \Omega \text{ and }\Sigma \coloneqq (0,\infty) \times \Gamma,
\end{gather*}
Additionally, as notations of base spaces, we let
\begin{gather*}
  H \coloneqq L^2(\Omega),\ V \coloneqq H^1(\Omega),\ \mathscr H_T \coloneqq L^2(0,T;H), \text{ and } \mathscr H \coloneqq L^2(0,\infty;H). 
\end{gather*}
\medskip

\noindent
\underline{\textbf{\textit{Notations in convex analysis.}}}
Let $X$ be an abstract Hilbert space $X$. For a proper, lower semi-continuous (l.s.c.), and convex function $\Psi : \,X \longrightarrow (-\infty, \infty]$ on a Hilbert space $X$, we denote by $D(\Psi)$ the effective domain of $\Psi$. Also, we denote by $\partial \Psi$ the subdifferential of $\Psi$. The set $D(\partial \Psi) := \left\{ z \in X\,|\, \partial \Psi(z) \neq \emptyset \right\}$ is called the domain of $\partial\Psi$. 
The subdifferential $\partial\Psi$ can be regarded as a generalized notion of derivative for a convex function $\Psi$, and it is known to be a maximal monotone graph in the product space $X\times X$.
We often use the notation ``$[z_0, z_0^*] \in \partial \Psi ~{\rm in}~ X \times X$", to mean that ``$z_0^* \in \partial \Psi(z_0) ~{\rm in}~ X~{\rm for}~ z_0 \in D(\partial \Psi)$", by identifying the operator $\partial \Psi$ with its graph in $X\times X$.
\medskip

\begin{example}\label{ex 1}
    Let $ \{ \gamma_\varepsilon \}_{\varepsilon \in [0, 1)} $ be the sequence of convex functions defined as follows:
  \begin{gather}
      \gamma_\varepsilon : y \in \mathbb R^N \mapsto \gamma_\varepsilon(y) \coloneqq \sqrt{\varepsilon^2 + |y|^2} \in [0,\infty), \mbox{ for $ \varepsilon \in [0, 1) $.}\label{gamma}\noeqref{gamma}
  \end{gather}
Then, the following two items hold.
    \begin{description}
        \item[\textmd{(\,I\,)}] Let $\{ \Phi_\varepsilon \}_{\varepsilon \in [0,1)}$ be a sequence of functionals on $[H]^N$, defined as:
      \begin{equation}
          \Phi_\varepsilon : {\bm w} \in [H]^N \mapsto \Phi_\varepsilon({\bm w}) := \int_\Omega \gamma_\varepsilon({\bm w}) \,dx \in [0,\infty], \mbox{ for $ \varepsilon \in [0, 1) $.}
      \end{equation}
            Then, for every $\varepsilon \in [0,1)$, $\Phi_\varepsilon$ is the proper l.s.c. and convex function, such that
      \begin{equation}
          D(\Phi_\varepsilon) = D(\partial \Phi_\varepsilon) = [H]^N,
      \end{equation}
      and
      \begin{gather}
        \partial \Phi_\varepsilon(\bm{w}) ~ := \left\{ \begin{array}{l}
            \bigl\{ \nabla \gamma_\varepsilon(\bm{w}) \bigr\}, \mbox{ if $ \varepsilon > 0 $,}
            \\[2ex]
            \autom{ \bm{w}^* \in [H]^N }{ \bm{w}^* \in \partial \gamma_0 (\bm{w}) \aein \Omega }, \text{ if } \varepsilon = 0,
        \end{array} \right.
        \\
          \mbox{in $[H]^N$, for any $ \bm{w} \in [H]^N $.}
    \end{gather}
\item[\textmd{\II}] Let $I \subset (0,T)$ be any open interval, and let $\{ \widehat{\Phi}_\varepsilon^I \}_{\varepsilon \in [0,1) }$ be a sequence of functionals on $L^2(I;[H]^N) \,( = \bigl[ L^2(I;H) \bigr]^N)$, defined as:
    \begin{equation}
        \widehat{\Phi}_\varepsilon^I : {\bm w} \in L^2(I;[H]^N) \mapsto \widehat{\Phi}_\varepsilon^I ({\bm w}) := \int_I \Phi_\varepsilon({\bm w}(t))\,dt \in [0,\infty], \mbox{ for $ \varepsilon \in [0, 1) $.}
    \end{equation}
    Then, for every $\varepsilon \in [0,1)$, $\widehat{\Phi}_\varepsilon^I$ is the proper l.s.c. and convex function, such that
    \begin{equation}
      D(\widehat{\Phi}_\varepsilon^I) = D(\partial \widehat{\Phi}_\varepsilon^I) = L^2(I;[H]^N),
    \end{equation}
    and
    \begin{align}
        \partial \widehat{\Phi}_\varepsilon^I({\bm w}) &= \bigl\{ \tilde{\bm w}^* \in L^2(I;[H]^N) \,|\, \tilde{\bm w}^* (t) \in \partial \Phi_\varepsilon^I({\bm w}(t)) \mbox{ in } [H]^N, \ \mbox{a.e. }t \in I \bigr\}
      \\
      &= \left\{ \begin{array}{l}
          \bigl\{ \nabla \gamma_\varepsilon(\bm{w}) \bigr\}, \mbox{ if $ \varepsilon \in (0,1) $,}
        \\[2ex]
        \autom{ \bm{w}^* \in L^2(I;[H]^N) }{ \bm{w}^* \in \partial_0 \gamma_\varepsilon( \bm{w} ) \aein I \times \Omega }, \text{ if } \varepsilon = 0,
    \end{array} \right.
    \\
    &\mbox{in $L^2(I;[H]^N)$, for any $ \bm{w} \in L^2(I;[H]^N) $.}
    \end{align}
    \end{description}
\end{example}
\medskip

\noindent
\underline{\textbf{\textit{Notations for the time-discretization.}}}
Let $\tau > 0$ be a constant of the time step-size, and let $\{ t_i \}_{i=0}^\infty \subset [0,\infty)$ be the time sequence defined as:
\begin{equation}
  t_i := i\tau,\ i=0,1,2,\ldots.
\end{equation}
Let $X$ be a Banach space. Then, for any sequence $\{ [t_i,z_i] \}_{i=0}^\infty \subset[0,\infty) \times X$, we define the \textit{forward time-interpolation} $[\overline{z}]_\tau \in L^\infty_\mathrm{loc}([0,\infty); X)$, the \textit{backward time-interpolation} $[\underline{z}]_\tau \in L^\infty_\mathrm{loc}([0,\infty);X)$ and the \textit{linear time-interpolation} $[z]_\tau \in W^{1,2}_\mathrm{loc}([0,\infty);X)$, by letting:
\begin{equation}
  \left\{\begin{aligned}
      &[\overline{z}]_\tau(t) := \chi_{(-\infty,0]}(t) z_0 + \sum_{i=1}^\infty \chi_{(t_{i-1}, t_i]}(t) z_i, \\
      &[\underline{z}]_\tau(t) := \chi_{(-\infty,0]}(t) z_0 + \sum_{i=0}^\infty \chi_{(t_{i}, t_{i+1}]}(t) z_{i}, \\
    &[z]_\tau (t) := \sum_{i=1}^\infty \chi_{[t_{i-1},t_i)}(t) \left(\frac{t-t_{i-1}}{\tau} z_{i} + \frac{t_i - t}{\tau} z_{i-1}\right),
  \end{aligned}\right. ~{\rm in}~ X,\ {\rm for}~ t \geq 0, \label{eq:time-interpolation}
\end{equation}
respectively.

In the meantime, for any $ q \in [1, \infty) $ and any $ \zeta \in L_\mathrm{loc}^q([0, \infty); X) $, we denote by $ \{ \zeta_i \}_{i = 0}^\infty \subset X $ the sequence of time-discretization data of $ \zeta $, defined as:
\begin{align}\label{tI01}
    & \zeta_0 := 0 \mbox{ in $X$, and }\zeta_i := \frac{1}{\tau} \int_{t_{i -1}}^{t_i} \zeta(\varsigma) \, d \varsigma ~ \mbox{ in $ X $, ~ for $ i = 1, 2, 3, \dots $.}
\end{align}
As is easily checked, the time-interpolations $ [\overline{\zeta}]_\tau, [\underline{\zeta}]_\tau \in L^q_\mathrm{loc}([0, \infty); X) $ for the above $ \{ \zeta_i \}_{i = 0}^\infty $ fulfill that:
\begin{align}\label{tI02}
    & [\overline{\zeta}]_\tau \to \zeta \mbox{ and } [\overline{\zeta}]_\tau \to \zeta \mbox{ in $ L^q_\mathrm{loc}([0, \infty); X) $, as $ \tau \downarrow 0 $.}
\end{align}

For any $w \in C([0, T]; X)$, the sequence $\{ w_i \}_{i = 0}^\infty \subset X$ is given as:
\begin{gather}
  w_i \coloneqq 
  \begin{cases}
    w(t_i), & \text{if } t_i \leq T,\\
    w(t_{i-1}), & \text{if } t_{i-1} \leq T < t_i.\\
    0, & \text{otherwise}.
  \end{cases}\label{interpolation}\noeqref{interpolation}
\end{gather}

\begin{definition}[Mosco-convergence: cf. \cite{MR0298508}]\label{Def.Mosco}
  Let $ X $ be an abstract Hilbert space. Let $ \Psi : X \rightarrow (-\infty, \infty] $ be a proper, l.s.c., and convex function, and let $ \{ \Psi_n \}_{n = 1}^\infty $ be a sequence of proper, l.s.c., and convex functions $ \Psi_n : X \rightarrow (-\infty, \infty] $, $ n = 1, 2, 3, \dots $.  Then, it is said that $ \Psi_n \to \Psi $ on $ X $, in the sense of Mosco, as $ n \to \infty $, iff. the following two conditions are fulfilled:
  \begin{description}
    \item[(\hypertarget{M_lb}{M1}) The condition of lower-bound:]$ \ds \varliminf_{n \to \infty} \Psi_n(\check{w}_n) \geq \Psi(\check{w}) $, if $ \check{w} \in X $, $ \{ \check{w}_n  \}_{n = 1}^\infty \subset X $, and $ \check{w}_n \to \check{w} $ weakly in $ X $, as $ n \to \infty $. 
    \item[(\hypertarget{M_opt}{M2}) The condition of optimality:]for any $ \hat{w} \in D(\Psi) $, there exists a sequence \linebreak $ \{ \hat{w}_n \}_{n = 1}^\infty  \subset X $ such that $ \hat{w}_n \to \hat{w} $ in $ X $ and $ \Psi_n(\hat{w}_n) \to \Psi(\hat{w}) $, as $ n \to \infty $.
  \end{description}
    Additionally, when Mosco-convergence is stated with respect to a continuous parameter $\varepsilon \in [0,1)$, it means that the corresponding Mosco-convergence holds for every sequence $\{\varepsilon_n\}_{n\in\N} \subset [0,1)$ converging to the specified limit of $\varepsilon$.
\end{definition}

\begin{remark}\label{Rem.MG}
  Let $ X $, $ \Psi $, and $ \{ \Psi_n \}_{n = 1}^\infty $ be as in Definition~\ref{Def.Mosco}. Then, the following facts hold.
  \begin{description}
    \item[(\hypertarget{Fact1}{Fact\,1})](cf. \cite[Theorem 3.5.3]{MR0773850} and \cite[Chapter 2]{Kenmochi81}) Let us assume that
    \begin{equation}\label{Mosco01}
      \Psi_n \to \Psi \mbox{ on $ X $, in the sense of  Mosco, as $ n \to \infty $,}
      \vspace{-1ex}
    \end{equation}
and
\begin{equation*}
\left\{ ~ \parbox{10cm}{
$ [w, w^*] \in X \times X $, ~ $ [w_n, w_n^*] \in \partial \Psi_n $ in $ X \times X $, $ n \in \N $,
\\[1ex]
$ w_n \to w $ in $ X $ and $ w_n^* \to w^* $ weakly in $ X $, as $ n \to \infty $.
} \right.
\end{equation*}
Then, it holds that:
\begin{equation*}
[w, w^*] \in \partial \Psi \mbox{ in $ X \times X $, and } \Psi_n(w_n) \to \Psi(w) \mbox{, as $ n \to \infty $.}
\end{equation*}
    \item[(\hypertarget{Fact2}{Fact\,2})](cf. \cite[Lemma 4.1]{MR3661429} and \cite[Appendix]{MR2096945}) Let $ d \in \mathbb{N} $ denote dimension constant, and let $  S \subset \R^d $ be a bounded open set. Then, under the Mosco-convergence as in \eqref{Mosco01}, a sequence $ \{ \widehat{\Psi}_n^S \}_{n = 1}^\infty $ of proper, l.s.c., and convex functions on $ L^2(S; X) $, defined as:
        \begin{equation*}
            w \in L^2(S; X) \mapsto \widehat{\Psi}_n^S(w) := \left\{ \begin{array}{ll}
                    \multicolumn{2}{l}{\ds \int_S \Psi_n(w(t)) \, dt,}
                    \\[1ex]
                    & \mbox{ if $ \Psi_n(w) \in L^1(S) $,}
                    \\[2.5ex]
                    \infty, & \mbox{ otherwise,}
                \end{array} \right. \mbox{for $ n = 1, 2, 3, \dots $;}
        \end{equation*}
        converges to a proper, l.s.c., and convex function $ \widehat{\Psi}^S $ on $ L^2(S; X) $, defined as:
        \begin{equation*}
            z \in L^2(S; X) \mapsto \widehat{\Psi}^S(z) := \left\{ \begin{array}{ll}
                    \multicolumn{2}{l}{\ds \int_S \Psi(z(t)) \, dt, \mbox{ if $ \Psi(z) \in L^1(S) $,}}
                    \\[2ex]
                    \infty, & \mbox{ otherwise;}
                \end{array} \right. 
        \end{equation*}
        on $ L^2(S; X) $, in the sense of Mosco, as $ n \to \infty $. 
\end{description}
\end{remark}

\begin{example}[Examples of Mosco-convergence]\label{Rem.ExMG}
    Let $ \varepsilon_0 \in (0,1) $ be arbitrary fixed constant, and $\{ \gamma_\varepsilon \}_{\varepsilon \in (0,1)}$ be as in \eqref{gamma}. Then, the following three items hold.
    \begin{description}
        \item[\textmd{(O)}] $\ds{\gamma_\varepsilon \to \gamma_{\varepsilon_0} \mbox{ on $ \R^N $, in the sense of Mosco, as $ \varepsilon \to \varepsilon_0 $.}}$
        \item[\textmd{(\,I\,)}] Let $\{ \Phi_\varepsilon \}_{\varepsilon \in (0,1)}$ be the sequence of proper l.s.c. and convex functions on $[H]^N$, as in Example \ref{ex 1} (I). Then, 
      \begin{equation}
        \Phi_\varepsilon \to \Phi_{\varepsilon_0} \mbox{ on $ [H]^N $, in the sense of Mosco, as $ \varepsilon \to \varepsilon_0 $.}
      \end{equation}
  \item[\textmd{\II}] Let $I \subset (0,T)$ be an open interval, and let $\{ \widehat{\Phi}_\varepsilon^I \}_{\varepsilon \in (0,1)}$ be the sequence of proper l.s.c. and convex functions on $L^2(I;[H]^N)$, as a Example \ref{ex 1} (II). Then, 
      \begin{equation}
        \widehat{\Phi}_\varepsilon^I \to \widehat{\Phi}_{\varepsilon_0}^I \mbox{ on $ L^2(I;[H]^N) $, in the sense of Mosco, as $ \varepsilon \to \varepsilon_0 $.}
      \end{equation}
    \end{description}
\end{example}

\setcounter{main theorem}{0}
\section{Main Results}

In this section, the main assertions are discussed under the following assumptions:
\begin{description}
  \item[{\boldmath(A1)}] $\alpha$ belongs to $L^\infty(Q)$ with a positive lower bound $\delta_\alpha > 0$. Moreover, the time derivative $\partial_t \alpha$ belongs to $ L^2(0,\infty;L^{q^*}(\Omega)) $ with
  \begin{gather*}
    q^* \coloneqq \frac{ 2 p^* }{ p^* - 2 } \text{\ \ and\ \ } p^* \coloneqq
    \begin{cases}
      4, \If N \leq 2,\\[1ex]
      \ds\frac{ 2N }{ N - 2 }, \otherwise.
    \end{cases}
  \end{gather*}
  \item[{\boldmath(A2)}] $\beta$ belongs to $W^{1,\infty}(Q)$ and, in this time-local setting, satisfies $\delta_T \coloneqq \inf \beta(Q_T) > 0$ for any $T \in (0,\infty)$. Additionally, $\partial_t \beta \leq 0$ a.e. in $Q$, i.e. $\beta$ is nonincreasing in time.
  \item[{\boldmath(A3)}] $g : \mathbb R \longrightarrow \mathbb R$ is a locally Lipschitz continuous function with a nonnegative primitive $G : \mathbb R \longrightarrow \mathbb R$, i.e. $g$ coincides with the derivative of $ G $.
  \item[{\boldmath(A4)}] The external force $ f \in L^\infty(Q) $ is given, and the initial data $u_0$ belongs to $V \cap L^\infty(\Omega)$.
  \item[{\boldmath(A5)}] There exists a constant $M > 0$ which is so large to satisfy that
      \begin{gather*}
        \norm{u_0}_{L^\infty(\Omega)} \leq M,\ g(M) - \norm{f}_{L^\infty(Q)} \geq 0, \text{ and } g(-M) + \norm{f}_{L^\infty(Q)} \leq 0.
      \end{gather*}
\end{description}

Next, we define the solutions to (P)$_\varepsilon$ and (S)$_\varepsilon$, respectively.
\begin{definition}[Solution to (P)$_\varepsilon$]\label{def 2}
    For any $\varepsilon \in [0,1)$, a function $u^\varepsilon : [0,\infty) \longrightarrow H$ is called a solution to (P)$_\varepsilon$, iff. the following conditions are satisfied:
\begin{description}
  \item[{\boldmath(S0)$_\varepsilon$}] $u^\varepsilon \in W^{1,2}_\loc([0,\infty);V)$, and $u^\varepsilon(0) = u_0$ in $H$.
 \item[{\boldmath(S1)$_\varepsilon$}] There exists a function $\bm{\omega}^*_\varepsilon \in L^\infty_\loc([0,\infty);[H]^N)$ such that
    \begin{gather*}
      \bm{\omega}^*_\varepsilon \in \partial \gamma_\varepsilon (\nabla u^\varepsilon), \text{ for}\aein Q,
    \end{gather*}
  and $u^\varepsilon$ solves the following variational identity:
    \begin{gather*}
        \bigl( \alpha(t) \partial_t u^\varepsilon(t) + g( u^\varepsilon(t) ), \varphi \bigr)_H + \bigl( \bm{\omega}^*_\varepsilon(t) + \beta(t) \nabla \partial_t u^\varepsilon(t), \nabla \varphi \bigr)_{[H]^N}\\[1ex]
        = \bigl( f(t), \varphi \bigr)_H, \text{ for any } \varphi \in V \text{ and} \ae t \in (0,\infty).
    \end{gather*}
\end{description}
\end{definition}
\begin{definition}[Solution to (S)$_\varepsilon$]\label{def 3}
  For any $\varepsilon \in [0,1)$, a function $u^\varepsilon_\infty \in BV(\Omega) \cap H$ is called a solution to (S)$_\varepsilon$, iff. the following conditions are satisfied:
\begin{description}
 \item[{\boldmath(S2)$_\varepsilon$}]
    $\begin{aligned}
      \int_\Omega \gamma_\varepsilon(D u^\varepsilon_\infty) - \int_\Omega \gamma_\varepsilon(D \varphi) \leq \bigl( f_\infty - g(u^\varepsilon_\infty), u^\varepsilon_\infty - \varphi \bigr)_H, \text{ for any } \varphi \in BV(\Omega) \cap H.
    \end{aligned}$
\end{description}
    Let $S_\infty^\varepsilon$ be the set of solutions to (S)$_\varepsilon$. Note that $S_\infty^\varepsilon$ may be empty if (S)$_\varepsilon$ has no solution, and may contain more than one element if its solution is not unique.
\end{definition}
\begin{definition}[$\omega$-limit set]\label{def 4}
    For any $\varepsilon \in [0,1)$ and any global-in-time solution $u^\varepsilon$ to (P)$_\varepsilon$, we define its $\omega$-limit set $\omega(u^\varepsilon)$ as the set of all accumulation points of the trajectory $\{u^\varepsilon(t)\}_{t \geq 0}$ as $t \to \infty$ by setting
\begin{gather*}
    \omega(u^\varepsilon) \coloneqq \autom{ u^\varepsilon_\infty \in BV(\Omega) \cap H }{
    \parbox{8cm}{
      \centering
      $u^\varepsilon(t_m) \to u^\varepsilon_\infty \In H \as m \to \infty$,
      for some $ \{t_m\}_{m \in \mathbb N} \subset (0,\infty) $ satisfying $ t_m \uparrow \infty$
    }
    }.
\end{gather*}
\end{definition}

We now state four main theorems concerning solutions to (P)$_\varepsilon$ and their large-time behavior. The first theorem establishes the existence and uniqueness of global-in-time solutions.
\begin{main theorem}[Existence and uniqueness of solution]
Under the assumptions (A1)--(A5), for any $\varepsilon \in [0,1)$, the problem (P)$_\varepsilon$ admits a unique solution. 
\end{main theorem}

Next, we state three main theorems on the large-time behavior of solutions to (P)$_\varepsilon$.
\begin{main theorem}[Large-time behavior in the basic form]
  Let us assume (A1)--(A5), and let us additionally assume:
\begin{description}
    \item[{\boldmath(A6)}] There exists a time-independent function $f_\infty \in H$ such that $f - f_\infty \in \mathscr H$.
\end{description}
  Then, for any $\varepsilon \in [0,1)$, it holds that
\begin{gather*}
    \omega(u^\varepsilon) \cap S^\varepsilon_\infty \neq \emptyset.
\end{gather*}
    In particular, the steady-state problem (S)$_\varepsilon$ admits at least one solution, and at least one $\omega$-limit point of $u^\varepsilon$ is a solution to (S)$_\varepsilon$.
\end{main theorem}

The assertion of Main Theorem 2 can be strengthened by imposing one additional condition.

\begin{main theorem}[Large-time behavior under exponential decay of $\beta$]
  In addition to (A1)--(A6), let us assume:
  \begin{description}
    \item[{\boldmath(A2)$^\prime$}] There exists a constant $C > 0$ such that
    $ \partial_t \left( e^{Ct} \beta \right) \leq 0 \aein Q $.
    In particular, $\beta$ decays exponentially in time, uniformly in $\Omega$.
  \end{description}
  Then, for any $\varepsilon \in [0,1)$, it holds that
  \begin{gather*}
    \omega(u^\varepsilon) \subset S^\varepsilon_\infty,
  \end{gather*}
  i.e., every $\omega$-limit point of $u^\varepsilon$ is a solution to the steady-state problem (S)$_\varepsilon$.
\end{main theorem}

In Main Theorem 3, we imposed the strong condition that $\beta$ decays exponentially in time. However, this condition on $\beta$ can be weakened by imposing additional conditions ensuring uniqueness for the steady-state problem (S)$_\varepsilon$. This leads to the following final result.

\begin{main theorem}[Large-time behavior when $G$ is strictly convex]
  In addition to (A1)--(A6), let us assume:
  \begin{description}
    \item[{\boldmath(A7)}]
    \begin{description}[leftmargin=3em]
      \item[{\boldmath(a-1)}] There exists a constant $C_\alpha > 0$ such that
      $\partial_t \bigl( (t+1)^{C_\alpha} \alpha \bigr) \geq 0$ a.e. in $Q$;
      
  \item[{\boldmath~(a-2)}] There exists a constant $C_\beta > 0$ such that
      $\partial_t \bigl( (t+1)^{C_\beta} \beta \bigr) \geq 0$ a.e. in $Q$;

      \item[{\boldmath~(a-3)}] For the derivative $g'$ of $g$, there exists a constant
      $C_{g^\prime} > 0$ such that $g^\prime \geq C_{g^\prime}$ a.e. in $\mathbb R$.
      In particular, $g$ is strictly increasing and its potential $G$ is strictly convex on $ \mathbb R $;

      \item[{\boldmath~(a-4)}] $f \in W^{1,2}_\loc([0,\infty);H)$, and
      $\sqrt{t+1}\partial_t f \in \mathscr H$.
    \end{description}
  \end{description}
  Then, for any $\varepsilon \in [0,1)$, there exists a function
  $u_\infty^\varepsilon \in BV(\Omega) \cap H$ such that
  \begin{gather}
    S_\infty^\varepsilon = \omega(u^\varepsilon)
    = \{u_\infty^\varepsilon\}.
  \end{gather}
  In other words, the steady-state problem (S)$_\varepsilon$ admits the unique solution
  $u_\infty^\varepsilon$, and
  $u^\varepsilon(t) \to u_\infty^\varepsilon$ in $H$ as $t \to \infty$.
\end{main theorem}

\section{Proof of Main Theorem 1}

The global-in-time solutions to (P)$_\varepsilon$ are constructed by the time-discretization method. To this end, we fix a time-step size $\tau \coloneqq T/n$ with $n \in \mathbb N$ and a relaxation parameter $\varepsilon \in (0,1)$. We then introduce the following time-discretization scheme (AP)$_\varepsilon^\tau$ as an approximation of (P)$_\varepsilon$.
\bigskip
\begin{description}
    \item[\textmd{(AP)$_\varepsilon^\tau$:}] To find a sequence $\{ u^\varepsilon_i \}_{i = 1}^n \subset V $ satisfying
  \begin{gather*}
    \begin{cases}
      \ds \frac{\alpha_i}{\tau} ( u^\varepsilon_i - u^\varepsilon_{i-1} ) - \diver \left( \nabla \gamma_\varepsilon ( \nabla u^\varepsilon_i ) + \frac{\beta_i}{\tau} \nabla ( u^\varepsilon_i - u^\varepsilon_{i-1} ) \right) + g( \mathcal T_M u^\varepsilon_i ) = f_i \In H,\\[2ex]
      \ds \nabla u^\varepsilon_i \cdot n_\Gamma = 0, \text{ for } \aeon \Gamma,\ \ \ \ u^\varepsilon_0 = u_0 \In H,\\
    \end{cases}\\
    \text{for } i=1,2,\dots,n,
  \end{gather*}
\end{description}
where $f_i$, $\alpha_i$, and $\beta_i$ are determined according to the discretization procedures given in \eqref{tI01} for $f_i$ and in \eqref{interpolation} for $\alpha_i$ and $\beta_i$.
For the variational analysis of the time-discretization scheme (AP)$_\varepsilon^\tau$, we introduce the following approximating free-energy functional $\widetilde{\mathcal F}_\varepsilon : H \longrightarrow (-\infty, \infty]$:
\begin{gather*}
  \widetilde{\mathcal F}_\varepsilon : H \ni u \mapsto \widetilde{\mathcal F}_\varepsilon(u) \coloneqq 
  \begin{cases}
    \ds \int_\Omega \gamma_\varepsilon ( D u ) + \int_\Omega \widetilde{G}_M(u)\,dx, \If u \in BV(\Omega),\\[2ex]
    \ds +\infty, \otherwise,
  \end{cases}
\end{gather*}
where $\widetilde{G}_M \in C^{1,1}(\mathbb R)$ is a nonnegative primitive of $g \circ \mathcal T_M \in W^{1,\infty}(\mathbb R)$, with $\widetilde{G}_M(0) = G(0)$.
\medskip

The solution to the scheme (AP)$^\tau_\varepsilon$ is defined as follows.

\begin{definition}\label{def 5}
    For any $\varepsilon \in (0,1)$, a sequence of functions $ \{ u^\varepsilon_i \}_{i = 1}^n \subset V $ is called a solution to the time-discretization scheme (AP)$^\varepsilon_\tau$ iff.: 
    \begin{gather*}
        \left(
            \frac{\alpha_i}{\tau}(u^\varepsilon_i-u^\varepsilon_{i-1})
            +g(\mathcal T_Mu^\varepsilon_i),\varphi
        \right)_H
        +
        \left(
            \nabla\gamma_\varepsilon(\nabla u^\varepsilon_i)
            +\frac{\beta_i}{\tau}\nabla(u^\varepsilon_i-u^\varepsilon_{i-1}),
            \nabla\varphi
        \right)_{[H]^N}
        =
        (f_i,\varphi)_H,
        \\
        \text{for any $\varphi\in V$ and $i=1,2,\dots, n$.}
    \end{gather*}
\end{definition}

To prove Main Theorem 1, we first establish several lemmas. 
We begin with the existence and uniqueness of a solution to the time-discretization scheme (AP)$^\varepsilon_\tau$.
\begin{lemma}\label{lem 3-1}
    Let $\varepsilon \in (0,1)$. Let $\overline{u}_0 \in V$ and $\overline{f} \in H$.
  Let $\delta_{\overline{\alpha}} > 0$ and $\delta_{\overline{\beta}} > 0$, and let
  $\overline{\alpha} \in L^\infty(\Omega)$ and
  $\overline{\beta} \in W^{1,\infty}(\Omega)$ be fixed functions such that
  $\overline{\alpha} \geq \delta_{\overline{\alpha}}$ and
  $\overline{\beta} \geq \delta_{\overline{\beta}}$ a.e. in $\Omega$, respectively.
  Then, there exists a constant $\overline{\tau} \in (0,1)$, depending only on $\norm{g^\prime}_{L^\infty(-M,M)}$ and $\delta_{\overline{\alpha}}$, such that, for any $\tau \in (0,\overline{\tau})$, the nonlinear elliptic problem
  \begin{flalign}
    \begin{cases}
      \ds \frac{\overline{\alpha}}{\tau} ( u - \overline{u}_0 ) - \diver \left( \nabla \gamma_\varepsilon ( \nabla u ) + \frac{\overline{\beta}}{\tau} \nabla ( u - \overline{u}_0 ) \right) + g( \mathcal T_M u ) = \overline{f} \In H,\\[2ex]
      \ds \nabla u \cdot n_\Gamma = 0, \text{ for } \aeon \Gamma,
    \end{cases}\label{3.1.1}\noeqref{3.1.1}
  \end{flalign}
admits a unique solution   $u \in V$ in the following sense:
  \begin{gather*}
    \left( \frac{\overline{\alpha}}{\tau} ( u - \overline{u}_0 ) + g( \mathcal T_M u ), \varphi \right)_H + \left( \nabla \gamma_\varepsilon( \nabla u ) + \frac{\overline{\beta}}{\tau} \nabla ( u - \overline{u}_0 ), \nabla \varphi \right)_{[H]^N} = \bigl( \overline{f}, \varphi \bigr)_H,\\
    \text{for any } \varphi \in V.
  \end{gather*}
\end{lemma}

\begin{proof}
  For any $\varepsilon \in (0,1)$, we define a functional
  $\Upsilon_\varepsilon : H \longrightarrow (-\infty, \infty]$ by
  \begin{gather*}
    \Upsilon_\varepsilon:H\ni z \mapsto \Upsilon_\varepsilon(z) \coloneqq
    \begin{cases}
      \ds \frac{1}{2\tau} \int_\Omega \overline{\alpha} | z - \overline{u}_0 |^2\,dx
      + \frac{1}{2\tau} \int_\Omega \overline{\beta} | \nabla (z - \overline{u}_0) |^2\,dx\\[2ex]
      \ds \hspace{2em} + \int_\Omega \gamma_\varepsilon(\nabla z)\,dx
      + \int_\Omega \widetilde{G}_M (z)\,dx
      - \int_\Omega \overline{f} z\,dx, \If z \in V,\\[2ex]
      +\infty, \otherwise.
    \end{cases}
  \end{gather*}
  It is easily seen that $\Upsilon_\varepsilon$ is proper, l.s.c., and coercive on $H$.
  Since $\widetilde{G}_M$ is not necessarily convex but is semi-convex
  ($\lambda$-convex), $\Upsilon_\varepsilon$ is not strictly convex in general.
  However, the semi-convexity of $\widetilde{G}_M$, together with the
    lower bound $ \delta_{\overline{\alpha}} > 0 $ of $\overline{\alpha}$, ensures that
  $\widetilde{G}_M(z) + ( \overline{\alpha}|z - \overline{u}_0|^2 )/2\tau$
  is strictly convex whenever $\tau \in (0,\overline{\tau})$, where
  $\overline{\tau}$ is chosen as
  \begin{gather}
    0 < \overline{\tau} \coloneqq
    \frac{\delta_{\overline{\alpha}}}
    {\norm{g^\prime}_{L^\infty(-M, M)}}.
  \end{gather}

  Thus, for any $\tau \in (0,\overline{\tau})$, $\Upsilon_\varepsilon$
  admits a unique minimizer. The Euler--Lagrange equation for this minimization
  problem coincides with the variational formulation of \eqref{3.1.1}.
  Therefore, the minimizer is the unique solution to \eqref{3.1.1}, which
  completes the proof.
\end{proof}

\begin{remark}[Solving method for (AP)$^\tau_\varepsilon$]\label{rem 2}
    Note that the previous lemma enables us to solve the time-discretization scheme (AP)$^\tau_\varepsilon$ inductively. Indeed, we take a constant $\tau^\circ \in (0,1)$ sufficiently small so that
    \begin{gather}
        0 < \tau^\circ < \frac{\delta_\alpha}{\norm{g^\prime}_{L^\infty(-M,M)}}.
        \label{3.1.2}\noeqref{3.1.2}
    \end{gather}
    Then, Lemma \ref{lem 3-1} can be applied inductively to determine the $i$-th time-discrete solution $u^\varepsilon_i \in V$ by setting
    \begin{gather*}
        \overline{u}_0 \coloneqq u^\varepsilon_{i-1},\quad
        \overline{f} \coloneqq f_i,\quad
        \overline{\alpha} \coloneqq \alpha_i,\quad
        \overline{\beta} \coloneqq \beta_i,
        \\
        \text{with the uniform lower bound $\delta_{\overline{\alpha}}=\delta_\alpha$
        for $i=1,2,\dots,n$.}
    \end{gather*}
\end{remark}

In the next lemma, we derive some estimates for $[\overline{u^\varepsilon}]_\tau, [\underline{u^\varepsilon}]_\tau$, and $\partial_t[{u^\varepsilon}]_\tau$. To this end, using \eqref{tI02}, we take $\widetilde{\tau} \in (0,\tau^\circ)$ sufficiently small so that
\begin{gather*}
  \tau\sum_{i=1}^n \norm{f_i}_H^2
  = \norm{[\overline{f}]_\tau}_{\mathscr H_T}^2
  \leq \norm{f}_{\mathscr H_T}^2 + 1,
  \text{ for any } \tau \in (0,\widetilde{\tau}).
\end{gather*}

\begin{lemma}\label{lem 3-2}
  Let $\varepsilon \in (0,1)$. Then, there exist constants
  $\tau_* \in (0,1)$, $C_1 > 0$, and $C_2 > 0$, independent of
  $\varepsilon$, $\tau$, and the time index $i$, such that
\begin{gather}
  \frac{1}{\tau} \sum_{i = 1}^n \norm{
        u^\varepsilon_i - u^\varepsilon_{i-1} 
    }_V^2
  \leq C_1 \bigl(
    \widetilde{\mathcal F}_\varepsilon(u_0)
    + \norm{f}_{\mathscr H_T}^2 + 1
  \bigr),
  \text{ for any } \tau \in (0, \tau_*),
  \notag
  \\[1ex]
  \norm{
      u^\varepsilon_i
  }_V^2 
  \leq C_2 \bigl(
    \widetilde{\mathcal F}_\varepsilon(u_0)
    + \norm{u_0}_V^2
    + \norm{f}_{\mathscr H_T}^2 +1
  \bigr), \text{ for } i = 1,2,\dots,n.
  \label{3.1.3}\noeqref{3.1.3}
\end{gather}
\end{lemma}

\begin{proof}
    First, in view of \eqref{3.1.2}, we set
\begin{gather}
  0 < \tau_* \coloneqq \min \left\{ \frac{1}{6}\tau^\circ, \widetilde{\tau} \right\}.
  \label{3.1.4}\noeqref{3.1.4}
\end{gather}
Testing (AP)$_\varepsilon^\tau$ by $u^\varepsilon_i-u^\varepsilon_{i-1}$, we obtain the following discrete energy inequality:
\begin{gather}
    \frac{1}{4\tau} \norm{ \sqrt{\alpha_i} (u^\varepsilon_i - u^\varepsilon_{i-1}) }_H^2
    + \frac{1}{4\tau} \norm{ \sqrt{\beta_i} \nabla (u^\varepsilon_i - u^\varepsilon_{i-1}) }_{[H]^N}^2
    + \widetilde{\mathcal F}_\varepsilon(u^\varepsilon_i)\notag\\[-1ex]
    \label{3.1.5}\\[-1ex]
    \leq \widetilde{\mathcal F}_\varepsilon(u^\varepsilon_{i-1})
    + \frac{\tau}{2\delta_\alpha}\norm{f_i}_H^2,
    \text{ for } i=1,2,\dots,n.\notag
\end{gather}
Indeed, the terms in the above variational inequality are estimated as follows:
\begin{gather*}
  \frac{1}{\tau} \bigl( \alpha_i ( u^\varepsilon_i - u^\varepsilon_{i-1} ),
  u^\varepsilon_i - u^\varepsilon_{i-1} \bigr)_H
  = \frac{1}{\tau}
  \norm{\sqrt{\alpha_i} ( u^\varepsilon_i - u^\varepsilon_{i-1} ) }_H^2,
\end{gather*}
\begin{gather*}
  \frac{1}{\tau}\bigl( \beta_i \nabla ( u^\varepsilon_i - u^\varepsilon_{i-1} ),
  \nabla (u^\varepsilon_i - u^\varepsilon_{i-1} ) \bigr)_{[H]^N}
  \geq \frac{1}{4\tau}
  \norm{\sqrt{\beta_i} \nabla ( u^\varepsilon_i - u^\varepsilon_{i-1} ) }_{[H]^N}^2,
\end{gather*}
\begin{gather*}
  \bigl( \nabla \gamma_\varepsilon(\nabla u^\varepsilon_i),
  \nabla ( u^\varepsilon_i - u^\varepsilon_{i-1} ) \bigr)_{[H]^N}
  \geq \int_\Omega \gamma_\varepsilon (\nabla u^\varepsilon_i)\,dx
  - \int_\Omega \gamma_\varepsilon (\nabla u^\varepsilon_{i-1})\,dx,
\end{gather*}
\begin{align*}
  \bigl( g(\mathcal T_M u^\varepsilon_i), u^\varepsilon_i - u^\varepsilon_{i-1} \bigr)_H
  &\geq \int_\Omega \widetilde{G}_M(u^\varepsilon_i)\,dx
  - \int_\Omega \widetilde{G}_M(u^\varepsilon_{i-1})\,dx
  - \frac{3}{2} \norm{g^\prime}_{L^\infty(-M,M)}
  \norm{u^\varepsilon_i - u^\varepsilon_{i-1}}_H^2\\
  &\geq \int_\Omega \widetilde{G}_M(u^\varepsilon_i)\,dx
  - \int_\Omega \widetilde{G}_M(u^\varepsilon_{i-1})\,dx
  - \frac{1}{4\tau}
  \norm{\sqrt{\alpha_i} ( u^\varepsilon_i - u^\varepsilon_{i-1} ) }_H^2,
\end{align*}
and
\begin{gather*}
  \bigl( f_i, u^\varepsilon_i - u^\varepsilon_{i-1} \bigr)_H
  \leq \frac{\tau}{2\delta_\alpha} \norm{f_i}_H^2
  + \frac{1}{2\tau}
  \norm{ \sqrt{\alpha_i} ( u^\varepsilon_i - u^\varepsilon_{i-1} ) }_H^2,
  \text{ for } i = 1,2,\dots,n.
\end{gather*}
    Summing \eqref{3.1.5} over $i=1,2\dots,n$, we have
\begin{align*}
    \frac{\delta_\alpha \land \delta_T}{4\tau}
    \sum_{i=1}^n \norm{ u^\varepsilon_i - u^\varepsilon_{i-1} }_V^2
    + \widetilde{\mathcal F}_\varepsilon(u^\varepsilon_n)
    &\leq \widetilde{\mathcal F}_\varepsilon(u_0)
    + \frac{\tau}{2\delta_\alpha} \sum_{i=1}^n \norm{f_i}_H^2\\[1ex]
    &\leq \widetilde{\mathcal F}_\varepsilon(u_0)
    + \frac{1}{2\delta_\alpha}
    \left( \norm{f}_{\mathscr H_T}^2 + 1 \right).
\end{align*}
Therefore, setting
$C_1 \coloneqq 4 / (\delta_\alpha^2 \land \delta_T^2 \land 1)$,
we obtain the first estimate in \eqref{3.1.3}.

  Next, applying H\"{o}lder's inequality, we can derive
\begin{align*}
  \norm{ u^\varepsilon_i }_V &\leq \norm{u_0}_V + \sum_{k=1}^n \norm{ u^\varepsilon_k - u^\varepsilon_{k-1} }_V \leq \norm{ u_0 }_V +  \left( \frac{1}{\tau} \sum_{k=1}^n \norm{ u^\varepsilon_k - u^\varepsilon_{k-1} }_V^2 \right)^{\frac{1}{2}} \left( \sum_{k=1}^n \tau \right)^{\frac{1}{2}}\\[1ex]
  &\leq \norm{u_0}_V + \sqrt{C_1 T} \left( \widetilde{\mathcal F}_\varepsilon(u_0) + \norm{f}_{\mathscr H_T}^2 + 1 \right)^{\frac{1}{2}}, \text{ for } i = 1, 2, \dots, n.
\end{align*}
 Then, by Young's inequality, we obtain
\begin{align*}
  \norm{u^\varepsilon_i}_V^2 &\leq 2\norm{u_0}_V^2 + 2C_1 T \left( \widetilde{\mathcal F}_\varepsilon(u_0) + \norm{f}_{\mathscr H_T}^2 + 1 \right)\\
  &\leq C_2 \left( \widetilde{\mathcal F}_\varepsilon(u_0) + \norm{u_0}_V^2 + \norm{f}_{\mathscr H_T}^2 + 1 \right), \text{ for } i=1,2,\dots,n,
\end{align*}
where
\begin{gather*}
  C_2 \coloneqq 2(C_1 T + 1).
\end{gather*}

This completes the proof of Lemma \ref{lem 3-2}.
\end{proof}

  Since $g$ is assumed to be only locally Lipschitz continuous in this work, we require the following lemma.

\begin{lemma}\label{lem 3-3}(Comparison principle)
  Let $\varepsilon \in (0,1)$. Let us assume the conditions  (A1)--(A5). Additionally, let us assume that $\widetilde{u}^\varepsilon_0 \in W^{1,2}_\loc([0,\infty);V)$ satisfies $\widetilde{u}^\varepsilon_0(0) = u_0 \In H$ and
  \begin{gather}
      \bigl( \alpha(t) \partial_t \widetilde{u}^\varepsilon_0(t) + g( \mathcal T_M \widetilde{u}^\varepsilon_0(t) ), \psi(t) \bigr)_H + \bigl( \nabla \gamma_\varepsilon(\nabla \widetilde{u}^\varepsilon_0(t)) + \beta(t) \nabla \partial_t \widetilde{u}^\varepsilon_0(t), \nabla \psi(t) \bigr)_{[H]^N}\\
      = \bigl( f(t), \psi(t) \bigr)_H, \text{ for any }\psi \in L^2_\loc([0,\infty);V) \text{ and}\ae t \in (0,\infty).\label{3.1.6}\noeqref{3.1.6}
  \end{gather}
  Then, it holds that $ |\widetilde{u}^\varepsilon_0| \leq M $ for a.e. in $Q$ holds, and hence, $\widetilde{u}^\varepsilon_0 \in L^\infty(Q)$.
\end{lemma}
\begin{proof}
 It suffices to show that
\begin{subequations}
  \begin{gather}
   \text{For any } T \in (0,\infty),\ \ 
   \begin{cases}
       \widetilde{u}^\varepsilon_0 \leq M \text{ for a.e. in } Q_T,\\
       -M \leq \widetilde{u}^\varepsilon_0 \text{ for a.e. in } Q_T.
   \end{cases}\\[-7ex]
   \label{3.1.7a}\\
   \label{3.1.7b}
  \end{gather}
  \end{subequations}

   First, we verify \eqref{3.1.7a}. Let us fix arbitrary $T \in (0,\infty)$. From (A5), it follows that
  \begin{gather}
      g(M) \geq  f \mbox{ a.e. in $ Q_T $,}
  \end{gather}
 and hence, 
    \begin{gather}
        \bigl( \alpha(t) \partial_t M + g(M), \psi(t) \bigr)_H + \bigl( \nabla \gamma_\varepsilon(\nabla M) + \beta(t) \nabla \partial_t M, \nabla \psi(t) \bigr)_{[H]^N}
        \geq \bigl( f(t), \psi(t) \bigr)_H,\\[1ex]
        \text{ for any } \psi \in L^2(0,T;V) \text{ with } \psi \geq 0 \text{ and} \ae t \in (0,T).\label{3.1.8}\noeqref{3.1.8}
  \end{gather}
We subtract \eqref{3.1.8} from \eqref{3.1.6}, and take $\psi \coloneqq [\widetilde{u}^\varepsilon_0 - M]^+$. Then, from the monotonicity of $\nabla \gamma_\varepsilon$, we have
\begin{gather*}
    \bigl( \alpha(t)\partial_t [\widetilde{u}^\varepsilon_0 - M]^+(t), [\widetilde{u}^\varepsilon_0 - M]^+(t) \bigr)_H + \bigl( \beta(t) \nabla \partial_t [\widetilde{u}^\varepsilon_0 - M]^+(t), \nabla [\widetilde{u}^\varepsilon_0 - M]^+(t) \bigr)_{[H]^N}\\[1ex]
    + \bigl( g( \mathcal T_M \widetilde{u}^\varepsilon_0(t) ) - g(M), [\widetilde{u}^\varepsilon_0 - M]^+(t) \bigr)_H \leq 0, \text{ for} \ae t \in (0,T).
  \end{gather*}
We next estimate the three terms on the left-hand side of the above inequality. The second and third terms are estimated as follows:
\begin{align}
  &\bigl( \beta(t) \nabla \partial_t [\widetilde{u}^\varepsilon_0 - M]^+(t), \nabla [\widetilde{u}^\varepsilon_0 - M]^+(t) \bigr)_{[H]^N}\\[1ex] 
  &\hspace{1em}\geq \frac{1}{2}\frac{d}{dt} \norm{ \sqrt{\beta(t)} \nabla [\widetilde{u}^\varepsilon_0 - M]^+(t)}_{[H]^N}^2 - \frac{ \norm{ \partial_t\beta }_{L^\infty(Q_T)} }{2(\delta_\alpha \land \delta_T)} \norm{ \sqrt{\beta(t)} \nabla [\widetilde{u}^\varepsilon_0 - M]^+(t)}_{[H]^N}^2,\hspace{1em}\label{3.1.9}\noeqref{3.1.9}
\end{align}
and
\begin{align}
  &- \bigl( g(\mathcal T_M \widetilde{u}^\varepsilon_0(t)) - g(M), [\widetilde{u}^\varepsilon_0 - M]^+(t) \bigr)_H\\
  & \hspace{2em} \leq \frac{\norm{g^\prime}_{L^\infty(-M,M)}}{\delta_\alpha \land \delta_T} \norm{ \sqrt{\alpha(t)} [\widetilde{u}^\varepsilon_0 - M]^+(t) }_H^2, \forae t \in (0,T).\label{3.1..10}\noeqref{3.1..10}
\end{align}
For the remaining first term, we proceed as follows:
\begin{align}
  &\bigl( \alpha(t) \partial_t [\widetilde{u}^\varepsilon_0 - M]^+(t), [\widetilde{u}^\varepsilon_0 - M]^+(t) \bigr)_H\\[1ex]
  &\hspace{1em} = \frac{1}{2}\frac{d}{dt} \norm{ \sqrt{\alpha(t)} [\widetilde{u}^\varepsilon_0 - M]^+(t) }_H^2 - \frac{1}{2} \int_\Omega \partial_t \alpha(t) | [\widetilde{u}^\varepsilon_0 - M]^+(t) |^2\,dx\\[1ex]
  &\hspace{1em} \geq \frac{1}{2}\frac{d}{dt} \norm{ \sqrt{\alpha(t)} [\widetilde{u}^\varepsilon_0 - M]^+(t) }_H^2 - \frac{C_V^{ L^{p^*} } }{2} \norm{ \partial_t \alpha(t) }_{ L^{q^*}(\Omega) } \norm{ [\widetilde{u}^\varepsilon_0 - M]^+(t) }_H \norm{ [\widetilde{u}^\varepsilon_0 - M]^+(t) }_V\\[1ex]
  &\hspace{1em} \geq \frac{1}{2}\frac{d}{dt} \norm{ \sqrt{\alpha(t)} [\widetilde{u}^\varepsilon_0 - M]^+(t) }_H^2\\
  &\hspace{2em} - \frac{ 3 C_V^{ L^{p^*} } \norm{ \partial_t \alpha(t)  }_{ L^{q^*}(\Omega) } }{4(\delta_\alpha \land \delta_T)} \left( \norm{ \sqrt{\alpha(t)} [\widetilde{u}^\varepsilon_0 - M]^+(t) }_H^2 + \norm{ \sqrt{\beta(t)} \nabla [\widetilde{u}^\varepsilon_0 - M]^+(t) }_{[H]^N}^2 \right),\\[1ex]
  &\hspace{25em} \forae t \in (0,T),\label{3.1..11}\noeqref{3.1..11}
\end{align}
where $C_V^{L^{p^*}}>0$ denotes an embedding constant for $V$ into $L^{p^*}(\Omega)$, valid for every dimension $N\in\mathbb{N}$.
    
    Taking into account \eqref{3.1.9}--\eqref{3.1..11}, it follows that
\begin{gather}
  \frac{d}{dt} X^\varepsilon(t) \leq \widetilde{C}_1 ( \norm{ \partial_t \alpha(t) }_{ L^{q^*}(\Omega) } + 1 ) X^\varepsilon(t), \forae t \in (0,T),\label{3.1..12}\noeqref{3.1..12}
\end{gather}
where
\begin{gather*}
  X^\varepsilon(t) \coloneqq \norm{ \sqrt{\alpha(t)} [\widetilde{u}^\varepsilon_0 - M]^+ (t)}_H^2 + \norm{ \sqrt{\beta(t)} \nabla [\widetilde{u}^\varepsilon_0 - M]^+ (t)}_{[H]^N}^2,\\
 \text{ for any } t \in [0,T] \text{ and } \varepsilon \in [0,1),
\end{gather*}
and
\begin{gather*}
  \widetilde{C}_1 \coloneqq \frac{3 C_V^{ p^* } + 2 \norm{ \partial_t \beta }_{L^\infty(Q_T)} + 4 \norm{g^\prime}_{L^\infty(-M,M)} }{ 2 (\delta_\alpha \land \delta_T) }.
\end{gather*}
Applying Gronwall's lemma in \eqref{3.1..12} together with (A5), one can see that
\begin{gather*}
  0 \leq X^\varepsilon(t) \leq \exp{ \left( \widetilde{C}_1 ( \norm{ \partial_t \alpha }_{ L^1( 0,T;L^{q^*}(\Omega) ) } + T) \right)} X^\varepsilon(0) = 0, \text{ for any } t \in [0,T].
\end{gather*}
This implies \eqref{3.1.7a}.

Secondly, to verify \eqref{3.1.7b}, we argue in the same manner as above with the following modifications. By (A5), we have
\begin{gather}
g(-M) \leq f \text{ a.e. in } Q_T,
\end{gather}
and hence
\begin{gather}
    \bigl( \alpha(t) \partial_t (-M) + g(-M), \psi(t) \bigr)_H + \bigl( \nabla \gamma_\varepsilon(\nabla (-M)) + \beta(t) \nabla \partial_t (-M), \nabla \psi(t) \bigr)_{[H]^N}\\[1ex]
        \leq \bigl( f(t), \psi(t) \bigr)_H, \text{ for any } \psi \in L^2(0,T;V) \text{ with } \psi \geq 0 \text{ and} \ae t \in (0,T).
  \end{gather}
Subtracting \eqref{3.1.6} from the above inequality and taking
$\psi \coloneqq [-\widetilde{u}^\varepsilon_0 - M]^+$,
we obtain the corresponding differential inequality for
$[-\widetilde{u}^\varepsilon_0 - M]^+$.
The remaining estimates are identical to those used in the proof of
\eqref{3.1.7a}. Therefore, Gronwall's lemma yields
\eqref{3.1.7b}.

Thus, we conclude the proof of Lemma \ref{lem 3-3}.
\end{proof}

\begin{proof}[Proof of Main Theorem 1]
    We prove Main Theorem 1 in the following three steps:
    \begin{description}
        \item[\textmd{(Step 1)}]Existence for $\varepsilon \in (0,1)$.
            \vspace{-1ex}
        \item[\textmd{(Step 2)}]Existence for $\varepsilon = 0$.
            \vspace{-1ex}
        \item[\textmd{(Step 3)}]Uniqueness for $\varepsilon \in [0,1)$.
    \end{description}

    \noindent
    \underline{\textit {(Step 1)\ Existence for $\mathit{\varepsilon \in (0,1)}$.}}\ \ Let $\varepsilon \in (0,1)$ and $T \in (0,\infty)$. Additionally, let $M > 0$ be a constant as in (A5).

     The preceding lemmas provide several uniform boundedness properties for the time interpolations of the time-discrete solutions to (AP)$^\tau_\varepsilon$. Indeed, by Lemma \ref{lem 3-2}, we have
\begin{itemize}
    \item[\textmd{$\flat$1)}] $\{ [u^\varepsilon]_\tau | \tau \in (0,\tau_*) \}$ is bounded in $L^\infty(0,T;V)$ and $W^{1,2}(0,T;V)$;
    \item[\textmd{$\flat$2)}] $\{ [\overline{u^\varepsilon}]_\tau | \tau \in (0,\tau_*) \}, \{ [\underline{u^\varepsilon}]_\tau | \tau \in (0,\tau_*) \}$ are bounded in $L^\infty(0,T;V)$.
\end{itemize}
    Thus, by the Aubin-type compactness theorem (cf. \cite[Corollary 4]{MR0916688}) and the Banach--Alaoglu theorem (cf. \cite[Section 1.2]{MR932730}), we can find a sequence $\{ \tau_m \}_{m\in\mathbb N} \subset (0, \tau_*)$ and a function $u^\varepsilon \in W^{1,2}(0,T;V)$ such that
\begin{gather}
  \tau_* \geq \tau_1 > \tau_2 > \cdots > \tau_m \downarrow 0 \as m \to \infty,
\end{gather}\vspace{-2ex}
\begin{gather}
  [u^\varepsilon]_{\tau_m} \to u^\varepsilon
  \begin{cases}
    & \In C([0,T];H),\\[1ex]
    \weakly & \In W^{1,2}(0,T;V),\\[1ex]
    \weaklystar & \In L^\infty(0,T;V),
  \end{cases}
\ \ \as m \to \infty,\label{3.1..13}\noeqref{3.1..13}
\end{gather}
and
\begin{gather}
  [\overline{u^\varepsilon}]_{\tau_m}, [\underline{u^\varepsilon}]_{\tau_m} \to u^\varepsilon \weaklystar \In L^\infty(0,T;V) \as m \to \infty.\label{3.1..14}\noeqref{3.1..14}
\end{gather}
We next show the following convergence of the forward and the backward time-interpolations:
\begin{gather}
  \begin{cases}
    [\overline{u^\varepsilon}]_{\tau_m}, [\underline{u^\varepsilon}]_{\tau_m} \to u^\varepsilon \In L^\infty(0,T;H),\\
    [\overline{u^\varepsilon}]_{\tau_m}(t), [\underline{u^\varepsilon}]_{\tau_m}(t) \to u^\varepsilon(t) \In H, \weakly \In V, \text{ for any } t \in [0,T],
  \end{cases}
  \hspace{-1em}\as m \to \infty.\hspace{2em}\label{3.1..15}\noeqref{3.1..15}
\end{gather}
By \eqref{3.1..13}, one can see that
\begin{align*}
  \norm{ [\overline{u^\varepsilon}]_{\tau_m}(t) - u^\varepsilon(t) }_H &\leq \norm{ [\overline{u^\varepsilon}]_{\tau_m}(t) - [u^\varepsilon]_{\tau_m}(t) }_H + \norm{ [u^\varepsilon]_{\tau_m}(t) - u^\varepsilon(t) }_H\\[1ex]
  &\hspace{-4em}\leq \tau_m^{\frac{1}{2}} \norm{ \partial_t [u^\varepsilon]_{\tau_m} }_{\mathscr H_T} + \norm{ [{u^\varepsilon}]_{\tau_m} - u^\varepsilon }_{C([0,T];H)} \to 0 \as m \to \infty, \text{ for any } t \in [0,T].
\end{align*}
This proves the strong convergences, while the weak convergence follows by combining it with Lemma \ref{lem 3-2}. Thus, \eqref{3.1..15} holds.

Now, we verify that $u^\varepsilon$ is a solution to (P)$_\varepsilon$. Let $I \subset (0,T)$ be an arbitrary open interval. Then, it follows from Lemma \ref{lem 3-1} that
\begin{gather}
  \int_I \bigl( [\overline{\alpha}]_{\tau_m}(t) \partial_t [u^\varepsilon]_{\tau_m}(t) + g(\mathcal T_M [\overline{u^\varepsilon}]_{\tau_m}(t)), \psi(t) \bigr)_H\,dt \\[1ex]
+ \int_I \bigl( [\overline{\beta}]_{\tau_m}(t) \nabla \partial_t [u^\varepsilon]_{\tau_m}(t) + \nabla \gamma_\varepsilon ( \nabla [\overline{u^\varepsilon}]_{\tau_m}(t)), \nabla \psi(t) \bigr)_{[H]^N}\,dt
    = \int_I \bigl( [\overline{f}]_{\tau_m}(t), \psi(t) \bigr)_H\,dt,\\ \text{ for any }\psi \in L^2(I;V).\label{3.1..16}\noeqref{3.1..16}
\end{gather}
To analyze the limit in \eqref{3.1..16} as $m \to \infty$, the following convergence is required.
\begin{gather}
    \nabla [\overline{u^\varepsilon}]_{\tau_m} \to \nabla u^\varepsilon \In [\mathscr H_T]^N \as m \to \infty.\label{3.1..17}\noeqref{3.1..17}
\end{gather}
  As a preliminary step toward proving \eqref{3.1..17}, we invoke (A1) and (A2), and take a subsequence $\{\tau_m\}_{m \in\mathbb N}$ (not relabeled) such that the following convergences hold as $m \to \infty$:
\begin{gather}
  \begin{cases}
    [\overline{\alpha}]_{\tau_m} \to \alpha,\ [\overline{\beta}]_{\tau_m} \to \beta \text{ and }\partial_t [\beta]_{\tau_m} \to \partial_t \beta \aeon Q,\\
      [\overline{\beta}]_{\tau_m}(t) \to \beta(t) \In L^\infty(\Omega), \text{ for any } t \in [0,\infty).
  \end{cases}\label{3.1..18}\noeqref{3.1..18}
\end{gather}
 Then, by (A2), we can compute as follows:
\begin{align*}
  &\int_0^s \bigl( [\overline{\beta}]_{\tau_m}(t) \nabla \partial_t [u^\varepsilon]_{\tau_m}(t), \nabla [\overline{u^\varepsilon}]_{\tau_m}(t) \bigr)_{[H]^N}\,dt\\[1ex]
  &\hspace{0.5em}= \sum_{i=1}^{m_s} \bigl( \beta_i \nabla ( u^\varepsilon_i - u^\varepsilon_{i-1}), \nabla u^\varepsilon_i \bigr)_{[H]^N} - \int^{m_s \tau_m}_s \bigl( [\overline{\beta}]_{\tau_m}(t) \nabla \partial_t [u^\varepsilon]_{\tau_m}(t), \nabla [\overline{u^\varepsilon}]_{\tau_m}(t) \bigr)_{[H]^N}\,dt\\[1ex]
  &\hspace{0.5em}\geq \frac{1}{2} \sum_{i=1}^{m_s} \left( \norm{ \sqrt{\beta_i} \nabla u^\varepsilon_i }_{[H]^N}^2 - \norm{ \sqrt{\beta_{i-1}} \nabla u^\varepsilon_{i-1} }_{[H]^N}^2 \right) + \frac{\tau_m}{2} \sum_{i=1}^{m_s} \int_\Omega \left( - \frac{\beta_i - \beta_{i-1}}{\tau_m} \right) |\nabla u^\varepsilon_{i-1}|^2\,dx\\[1ex]
    &\hspace{2em} - \tau_m^{\frac{1}{2}} \norm{\beta}_{L^\infty(Q_T)}\sup_{m \in \mathbb N}\norm{ \nabla [\overline{u^\varepsilon}]_{\tau_m} }_{L^\infty(0,T;[H]^N)} \sup_{m\in\mathbb N} \norm{\nabla \partial_t [u^\varepsilon]_{\tau_m} }_{[\mathscr H_T]^N}\\[1ex]
  &\hspace{0.5em} \geq \frac{1}{2} \left( \norm{ \sqrt{ \vphantom{\rule[-0.2ex]{0pt}{2.1ex}}\smash{ [\overline{ \beta}]_{\tau_m}(s) } } \nabla [\overline{u^\varepsilon}]_{\tau_m}(s) }_{[H]^N}^2 - \norm{ \sqrt{\beta(0)} \nabla u_0 }_{[H]^N}^2 \right) + \frac{1}{2} \norm{ \sqrt{ -\partial_t [\beta]_{\tau_m} } \nabla [\underline{u^\varepsilon}]_{\tau_m} }_{[\mathscr H_s]^N}^2\\[1ex]
    &\hspace{2em} - \tau_m^{\frac{1}{2}} \norm{\beta}_{L^\infty(Q_T)}\sup_{m \in \mathbb N}\norm{ \nabla [\overline{u^\varepsilon}]_{\tau_m} }_{L^\infty(0,T;[H]^N)} \sup_{m\in\mathbb N} \norm{\nabla \partial_t [u^\varepsilon]_{\tau_m} }_{[\mathscr H_T]^N}, \text{ for any } m \in \mathbb N.
\end{align*}
Therefore, with \eqref{3.1..15} and \eqref{3.1..18} in mind, one can see that
\begin{align}
  &\liminf_{m \to \infty}\int_0^s \bigl( [\overline{\beta}]_{\tau_m}(t) \nabla \partial_t [u^\varepsilon]_{\tau_m}(t), \nabla [\overline{u^\varepsilon}]_{\tau_m}(t) \bigr)_{[H]^N}\,dt\\[1ex]
    & \geq \liminf_{m \to \infty} \left[
        \frac{1}{2} \left( \norm{ \sqrt{ \vphantom{\rule[-0.2ex]{0pt}{2.1ex}}\smash{ [\overline{ \beta}]_{\tau_m}(s) } } \nabla [\overline{u^\varepsilon}]_{\tau_m}(s) }_{[H]^N}^2 \hspace{-0.5em}- \norm{ \sqrt{\beta(0)} \nabla u_0 }_{[H]^N}^2 \right) + \frac{1}{2} \norm{ \sqrt{ -\partial_t [\beta]_{\tau_m} } \nabla [\underline{u^\varepsilon}]_{\tau_m} }_{[\mathscr H_s]^N}^2\right.\\[-1.5ex]
    &\hspace{2em} \left. - \tau_m^{\frac{1}{2}} \norm{\beta}_{L^\infty(Q_T)}\sup_{m \in \mathbb N}\norm{ \nabla [\overline{u^\varepsilon}]_{\tau_m} }_{L^\infty(0,T;[H]^N)} \sup_{m\in\mathbb N} \norm{\nabla \partial_t [u^\varepsilon]_{\tau_m} }_{[\mathscr H_T]^N} \right]\\[1ex]
  &\hspace{1em} \geq \frac{1}{2} \left( \norm{ \sqrt{\beta(s)}  \nabla u^\varepsilon(s) }_{[H]^N}^2 - \norm{ \sqrt{\beta(0)} \nabla u_0 }_{[H]^N}^2 \right) + \frac{1}{2} \norm{ \sqrt{ -\partial_t \beta } \nabla u^\varepsilon }_{[\mathscr H_s]^N}^2 - 0\\[1ex]
  &\hspace{1em} = \int_0^s \bigl( \beta(t) \nabla \partial_t u^\varepsilon(t), \nabla u^\varepsilon(t) \bigr)_{[H]^N}\,dt,\label{3.1..19}\noeqref{3.1..19}
\end{align}
via
\begin{gather}
  \begin{cases}
      \ds \liminf_{m \to \infty}\norm{ \sqrt{ \vphantom{\rule[-0.2ex]{0pt}{2.1ex}}\smash{ [\overline{ \beta}]_{\tau_m}(s) } } \nabla [\overline{u^\varepsilon}]_{\tau_m}(s) }_{[H]^N}^2 \geq \norm{ \sqrt{\beta(s)}  \nabla u^\varepsilon(s) }_{[H]^N}^2, \\[1ex]
      \ds \liminf_{m \to \infty}\norm{ \sqrt{ -\partial_t [\beta]_{\tau_m} } \nabla [\underline{u^\varepsilon}]_{\tau_m} }_{[\mathscr H_s]^N}^2 \geq \norm{ \sqrt{ -\partial_t \beta } \nabla u^\varepsilon }_{[\mathscr H_s]^N}^2.
  \end{cases}\label{kakuten}\noeqref{kakuten}
\end{gather}
Also, from \eqref{3.1..14}, one can see that
\begin{gather}
    \liminf_{m \to \infty}\widehat{\Psi}^{(0,s)}_\varepsilon(\nabla [\overline{u^\varepsilon}]_{\tau_m} ) 
=\liminf_{m \to \infty} \int_0^s \int_\Omega \gamma_\varepsilon ( \nabla [\overline{u^\varepsilon}]_{\tau_m}(t) )\,dxdt
\\
    \geq \widehat{\Psi}^{(0,s)}_\varepsilon(\nabla u^\varepsilon ) = \int_0^s \int_\Omega \gamma_\varepsilon ( \nabla u^\varepsilon(t) )\,dxdt,\label{3.1..20}\noeqref{3.1..20}
\end{gather}
where $ \widehat{\Psi}^{(0,s)}_\varepsilon $ is the convex function on $ L^2(0, s; [H]^N) $ given in Example \ref{ex 1} (with $ I = (0, s) $).
On the other hand, taking $\psi \coloneqq [\overline{u^\varepsilon}]_{\tau_m} - u^\varepsilon$ as a test function in \eqref{3.1..16}, and invoking Example \ref{ex 1} together with \eqref{3.1..13}--\eqref{3.1..15} and \eqref{3.1..18}, we obtain
\begin{align}
    &\limsup_{m \to \infty} \left( \int_0^s \bigl( [\overline{\beta}]_{\tau_m}(t) \nabla \partial_t [u^\varepsilon]_{\tau_m}(t), \nabla [\overline{u^\varepsilon}]_{\tau_m}(t) \bigr)_{[H]^N}\,dt + \widehat{\Psi}^{(0,s)}_\varepsilon(\nabla [\overline{u^\varepsilon}]_{\tau_m} ) \right)\\[1ex]
    & \hspace{1em} \leq \limsup_{m \to \infty} \left( \int_0^s \hspace{-0.2em} \bigl( [\overline{\beta}]_{\tau_m}(t) \nabla \partial_t [u^\varepsilon]_{\tau_m}(t), \nabla u^\varepsilon(t) \bigr)_{[H]^N}\,dt + \widehat{\Psi}^{(0,s)}_\varepsilon(\nabla u^\varepsilon ) \right.\\[1ex]
    & \hspace{2em} \left. - \int_0^s \bigl( [\overline{\alpha}]_{\tau_m}(t) \partial_t [u^\varepsilon]_{\tau_m}(t) + g(\mathcal T_M [\overline{u^\varepsilon}]_{\tau_m}(t) - [\overline{f}]_{\tau_m}(t), [\overline{u^\varepsilon}]_{\tau_m}(t) - u^\varepsilon(t) \bigr)_H\,dt \right)\\[1ex]
    & \hspace{1em} = \int_0^s \bigl( \beta(t) \nabla \partial_t u^\varepsilon(t), \nabla u^\varepsilon(t) \bigr)_{[H]^N}\,dt + \widehat{\Psi}^{(0,s)}_\varepsilon(\nabla u^\varepsilon ).\label{3.1..21}\noeqref{3.1..21}
\end{align}
Therefore, by applying \textbf{Fact($\ast$)}, we obtain the following convergences as $m \to \infty$:
\begin{gather}
    \int_0^s \bigl( [\overline{\beta}]_{\tau_m}(t) \nabla \partial_t [u^\varepsilon]_{\tau_m}(t), \nabla [\overline{u^\varepsilon}]_{\tau_m}(t) \bigr)_{[H]^N}\,dt \to \int_0^s \bigl( \beta(t) \nabla \partial_t u^\varepsilon(t), \nabla u^\varepsilon(t) \bigr)_{[H]^N}\,dt,\hspace{2em}\label{3.1..22}\noeqref{3.1..22}
\end{gather}
and
\begin{gather}
    \widehat{\Psi}^{(0,s)}_\varepsilon(\nabla [\overline{u^\varepsilon}]_{\tau_m} ) \to \widehat{\Psi}^{(0,s)}_\varepsilon(\nabla u^\varepsilon ).\label{gam}\noeqref{gam}
\end{gather}
In particular, \eqref{3.1..19} and \eqref{3.1..22} enable us to say that
\begin{gather}
    \limsup_{m \to \infty}\Bigl( \norm{ \sqrt{ \vphantom{\rule[-0.2ex]{0pt}{2.1ex}}\smash{ [\overline{ \beta}]_{\tau_m}(s) } } \nabla [\overline{u^\varepsilon}]_{\tau_m}(s) }_{[H]^N}^2 + \norm{ \sqrt{ -\partial_t [\beta]_{\tau_m} } \nabla [\underline{u^\varepsilon}]_{\tau_m} }_{[\mathscr H_s]^N}^2 \Bigr)\\[1ex]
    \leq \norm{ \sqrt{\beta(s)}  \nabla u^\varepsilon(s) }_{[H]^N}^2 + \norm{ \sqrt{ -\partial_t \beta } \nabla u^\varepsilon }_{[\mathscr H_s]^N}^2.\label{limsup}\noeqref{limsup}
\end{gather}
Applying \textbf{Fact($\ast$)} again to \eqref{3.1..22} and \eqref{limsup}, and taking into account \eqref{3.1..15} and \eqref{3.1..18} and the uniform convexity of $L^2$-based topology, we arrive at
\begin{gather}
    \sqrt{ \vphantom{\rule[-0.2ex]{0pt}{2.3ex}}\smash{ [\overline{ \beta}]_{\tau_m}(s) } } \nabla [\overline{u^\varepsilon}]_{\tau_m}(s) \to \sqrt{\beta(s)} \nabla u^\varepsilon(s) \In [H]^N \as m \to \infty,\\[1ex]
    \text{ for any } s \in [0,T].\label{kakuten2}\noeqref{kakuten2}
\end{gather}
The convergence \eqref{3.1..17} will be verified as a consequence of the pointwise convergence of \eqref{3.1..18}, \eqref{kakuten2}, the Bochner--Lebesgue dominated convergence theorem, and (A2).

In view of \eqref{3.1..13}--\eqref{3.1..15} and \eqref{3.1..17}, letting $m \to \infty$ in \eqref{3.1..16} yields 
\begin{gather}
  \int_I \bigl( \alpha(t) \partial_t u^\varepsilon(t) + g(\mathcal T_M u^\varepsilon(t)), \psi(t) \bigr)_H\,dt + \int_I \bigl( \beta(t) \nabla \partial_t u^\varepsilon(t) + \nabla \gamma_\varepsilon ( \nabla u^\varepsilon(t)), \nabla \psi(t) \bigr)_{[H]^N}\,dt\\[1ex]
    = \int_I \bigl( f(t), \psi(t) \bigr)_H\,dt, \text{ for any }\psi \in L^2( I; V) \text{ and any open interval }I \subset (0,T).\hspace{2em}\label{positive ep}\noeqref{positive ep}
\end{gather}
Furthermore, applying Lemma \ref{lem 3-3} with $\widetilde{u}^\varepsilon_0 \coloneqq u^\varepsilon \In H$, we have $\norm{u^\varepsilon}_{L^\infty(Q_T)} \leq M$. This $L^\infty$ estimate together with arbitrariness of $I \subset (0,T)$ shows that $u^\varepsilon$ is a solution to (P)$_\varepsilon$ for any $\varepsilon \in (0,1)$.
\medskip\\

\noindent\underline{\textit{(Step 2) Existence for $\mathit{\varepsilon = 0}$.}}\ \ Based on \eqref{positive ep} for $\varepsilon \in (0,1)$, we consider the limit as $\varepsilon \to 0$. Applying Lemma \ref{lem 3-2} together with \eqref{3.1..13} and \eqref{3.1..14}, we have
\begin{gather*}
  \norm{ \partial_t u^\varepsilon }_{L^2(0,T;V)}^2 \leq \liminf_{m \to \infty} \norm{ \partial_t [u^\varepsilon]_{\tau_m} }_{L^2(0,T;V)}^2 \leq C_1 \bigl( \widetilde{\mathcal F}_\varepsilon(u_0) + \norm{f}_{\mathscr H_T}^2 + 1 \bigr), \text{ for any } \varepsilon \in (0,1) ,\\[1ex]
  \norm{ u^\varepsilon (t) }_{V}^2 \leq \liminf_{m \to \infty} \norm{ [\overline{u^\varepsilon}]_{\tau_m}(t) }_{V}^2 \leq C_2 \left( \widetilde{\mathcal F}_\varepsilon(u_0) + \norm{u_0}_{V}^2 + \norm{f}_{\mathscr H_T}^2 + 1 \right),\\
  \hspace{22em} \text{ for any } t \in [0,T] \text{ and } \varepsilon \in (0,1),
\end{gather*}
and in particular,
\begin{gather*}
  \widetilde{\mathcal F}_\varepsilon(u_0) = \int_\Omega \gamma_\varepsilon(\nabla u_0)\,dx + \int_\Omega \widetilde{G}_M(u_0)\,dx \leq \mathcal L^N(\Omega) + \norm{\nabla u_0}_{[L^1(\Omega)]} + \norm{G(u_0)}_{L^1(\Omega)} < \infty,\\
  \text{for any }\varepsilon \in (0,1).
\end{gather*}
The above estimates derive the following uniform boundedness:
\begin{itemize}
  \item[$\flat$3)] $ \{ u^\varepsilon \}_{\varepsilon \in (0,1)} $ is bounded in $ L^\infty(0,T;V) $ and $ W^{1,2}(0,T;V) $.
\end{itemize}
Therefore, by the Aubin-type compactness theorem (cf. \cite[Corollary 4]{MR0916688}) and the Banach--Alaoglu theorem (cf. \cite[Section 1.2]{MR932730}), we can find a sequence $\{ \varepsilon_m \}_{m \in \mathbb N} \subset (0,1)$, a function $u^0 \in W^{1,2}(0,T;V)$, and a vectorial function $\bm{\omega^*_0} \in [L^\infty(Q_T)]^N$ such that
\begin{gather*}
  \varepsilon_1 > \varepsilon_2 > \cdots > \varepsilon_m \downarrow 0 \as m \to \infty,
\end{gather*}
\begin{gather}
  u^{\varepsilon_m} \to u^0
  \begin{cases}
    & \In C([0,T];H),\\[1ex]
    \weakly & \In W^{1,2}(0,T;V),\\[1ex]
    \weaklystar & \In L^\infty(0,T;V),
  \end{cases}
\ \ \as m \to \infty, \label{3.1..24}\noeqref{3.1..24}
\end{gather}
and
\begin{gather}
    \nabla \gamma_{\varepsilon_m}(\nabla u^{\varepsilon_m}) \eqqcolon \bm{\omega^*_{\varepsilon_m}} \to \bm{\omega^*_0}
\weaklystar \In [L^\infty(Q_T)]^N
     \as m \to \infty.\label{3.1..25}\noeqref{3.1..25}
\end{gather} 
Also, the $L^\infty$-estimate of $\{ u^{\varepsilon_m}\}_{m \in \mathbb N}$ leads to 
\begin{gather}
    \norm{u^0}_{L^\infty(Q)} \leq \liminf_{m \to \infty} \norm{u^{\varepsilon_m}}_{L^\infty(Q)} \leq \sup_{\varepsilon \in [0,\infty)} \norm{u^\varepsilon}_{L^\infty(Q)} \leq M.\label{Linfty}\noeqref{Linfty}
\end{gather}
Additionally, just as in the verification of \eqref{3.1..17}, it is inferred that
\begin{gather}
    \nabla u^{\varepsilon_m} \to \nabla u^0 \In [\mathscr H_T]^N \as m \to \infty.\label{grad0}\noeqref{grad0}
\end{gather}

Now, taking into account \eqref{3.1..24}, \eqref{3.1..25}, and \eqref{grad0}, and passing to the limit as $m \to \infty$ in \eqref{positive ep} with $\varepsilon \coloneqq \varepsilon_m$, we arrive at
\begin{gather}
  \int_I \bigl( \alpha(t) \partial_t u^0(t) + g( u^0(t)), \psi(t) \bigr)_H\,dt + \int_I \bigl( \beta(t) \nabla \partial_t u^0(t) + \bm{\omega}^*_0(t), \nabla \psi(t) \bigr)_{[H]^N}\,dt\\[1ex]
    = \int_I \bigl( f(t), \psi(t) \bigr)_H\,dt, \text{ for any }\psi \in L^2(I;V) \text{ and any open interval } I \subset (0,T).\hspace{2em}\label{0 ep}\noeqref{0 ep}
\end{gather}
Moreover, from \eqref{3.1..25}, \eqref{grad0} and \textbf{(Fact 1)}, one can observe that 
\begin{gather}
    \bm{\omega}^*_0 \in \partial \gamma_0 (\nabla u^0) \ae \In Q_T.\label{graph}\noeqref{graph}
\end{gather}
\eqref{0 ep} and \eqref{graph} implies that $u^0 \in W^{1,2}(0,T;V)$ is a solution to (P)$_0$.
\medskip\\

\noindent\underline{\textit{(Step 3) Uniqueness for $\mathit{\varepsilon \in[0,1)}$.}}\ \ Let $u^\varepsilon_1, u^\varepsilon_2 \in W^{1,2}(0,T;V)$ be two solutions. In view of Lemma \ref{lem 3-3} and \eqref{Linfty}, we know that
\begin{gather}
    \sup_{\varepsilon \in [0,1)} \norm{u^\varepsilon_k}_{L^\infty(Q)} \leq M, \text{ for } k= 1,2.\label{Linftyb}\noeqref{Linftyb}
\end{gather}
Keeping the monotonicity of $\partial \gamma_\varepsilon$ in mind, we subtract the variational identities (S1)$_\varepsilon$ for $u^\varepsilon_1$ and $u^\varepsilon_2$, and take $\varphi \coloneqq (u^\varepsilon_1 - u^\varepsilon_2)(t)$. Then, we have
\begin{gather}
  \bigl( \alpha(t) \partial_t (u^\varepsilon_1-u^\varepsilon_2)(t), (u^\varepsilon_1-u^\varepsilon_2)(t) \bigr)_H + \bigl( \beta(t) \nabla \partial_t (u^\varepsilon_1-u^\varepsilon_2)(t), \nabla (u^\varepsilon_1-u^\varepsilon_2)(t) \bigr)_{[H]^N}\\[1ex]
  \leq - \bigl( g(u^\varepsilon_1(t)) - g(u^\varepsilon_2(t)), (u^\varepsilon_1-u^\varepsilon_2)(t) \bigr)_H, \forae t \in (0,T).\label{3.1..26}\noeqref{3.1..26}
\end{gather}
Furthermore, using \eqref{Linftyb} and following the argument in the proof of Lemma \ref{lem 3-3}, we arrive at
\begin{gather*}
  0 \leq J^\varepsilon(t) \leq \exp{ \left( \widetilde{C}_1 ( \norm{ \partial_t \alpha }_{ L^1( 0,T;L^{q^*}(\Omega) ) } + T) \right)} J^\varepsilon(0) = 0, \text{ for any } t \in [0,T],
\end{gather*}
where
\begin{gather*}
  J^\varepsilon(t) \coloneqq \norm{ \sqrt{\alpha(t)} (u^\varepsilon_1 - u^\varepsilon_2) (t)}_H^2 + \norm{ \sqrt{\beta(t)} \nabla (u^\varepsilon_1 - u^\varepsilon_2) (t)}_{[H]^N}^2,\\
 \text{ for any } t \in [0,T] \text{ and } \varepsilon \in [0,1),
\end{gather*}
and $\widetilde{C}_1 > 0$ is a constant as in \eqref{3.1..12}.
This implies that the solution to the problem (P)$_\varepsilon$ is unique.

Thus, the proof of Main Theorem 1 is complete.
\end{proof}

\begin{remark}[Convergence of solutions as $ \varepsilon \to 0 $]\label{eps-dep}
    As a further consequence of the proof of Main Theorem 1, we obtain the convergence of the sequence of solutions $\{u^\varepsilon\}_{\varepsilon\in(0,1)}$ as $\varepsilon\to0$. Indeed, in the proof of Main Theorem 1, (Step 2) establishes this convergence along a subsequence in the senses specified in \eqref{3.1..24}--\eqref{grad0}, together with the limiting relations \eqref{0 ep} and \eqref{graph}. On the other hand, (Step 3) guarantees the uniqueness of the solution $u^0$ to the limiting problem (P)$_0$. Therefore, the convergence obtained in (Step 2) actually holds without passing to a subsequence. Namely, as $\varepsilon\to0$, the solution $u^\varepsilon$ converges to $u^0$ in the senses of \eqref{3.1..24}--\eqref{grad0}, without extracting a subsequence.
\end{remark}

\begin{remark}[$H^2$-regularity of solutions]
    By imposing additional regularity assumptions on the boundary $\Gamma$ and the initial data $u_0$, the time-decay condition $\partial_t\beta\leq 0$ as in (A2) is no longer necessary. Indeed, if $\Gamma$ is of class $C^4$ and $u_0$ belongs to the domain of the Laplacian with zero-Neumann boundary condition, then the regularity of the solution can be improved to the class $L^\infty_{\mathrm{loc}}([0,\infty);H^2(\Omega))$ by following an argument similar to that used in \cite{aiki2023class}. Due to this stronger regularity, we can obtain the strong convergence stated in \eqref{3.1..17} without using the time-decay property of $\beta$. Consequently, in this case, the conclusion of Main Theorem 1 remains valid without imposing the time-decay condition on $\beta$.
\end{remark}
\section{Proofs of Main Theorem 2--4}

\subsection{Proof of Main Theorem 2}
Before the proof of Main Theorem 2, we prove the following lemma.

\begin{lemma}\label{lem 4-1}(Energy inequality)
  Let us assume the conditions (A1)--(A6). For any $\varepsilon \in [0,1)$, let $u^\varepsilon$ be the unique solution to the problem (P)$_\varepsilon$. Then, the following energy inequality holds:
\begin{gather}
  \frac{1}{2} \int_s^t \int_\Omega \alpha(\sigma) | \partial_t u^\varepsilon(\sigma) |^2\,dx d\sigma + \frac{1}{2} \int_s^t \int_\Omega \beta(\sigma) | \nabla \partial_t u^\varepsilon(\sigma) |^2\,dx d\sigma + \mathcal F_\varepsilon (u^\varepsilon(t))\\[1ex]
  \leq \frac{1}{\delta_\alpha} \int_s^t \int_\Omega | f(\sigma) - f_\infty |^2\,dx d\sigma + \mathcal F_\varepsilon(u^\varepsilon(s)), \text{ for any } s,t \in [0,\infty) \text{ with } s\leq t.\label{4.1.1}\noeqref{4.1.1}
\end{gather}
Moreover, the following properties hold:
\begin{itemize}
  \item[$\sharp$1)] $ \partial_t u^\varepsilon \in \mathscr H $;
  \item[$\sharp$2)] $ \sqrt{\beta} \nabla \partial_t u^\varepsilon \in [\mathscr H]^N $;
  \item[$\sharp$3)] $ u^\varepsilon \in L^\infty(0,\infty; W^{1,1}(\Omega)) \cap L^\infty(Q)$.
\end{itemize}
\end{lemma}
\begin{proof}
    By taking $\varphi \coloneqq \partial_t u^\varepsilon(t)$ as a test function in (S1)$_\varepsilon$, we have

\begin{gather}
  \int_\Omega \alpha(t) | \partial_t u^\varepsilon(t) |^2\,dx + \int_\Omega \beta(t) | \nabla \partial_t u^\varepsilon(t) |^2\,dx + \int_\Omega \bm{\omega}^*_\varepsilon(t) \cdot \nabla \partial_t u^\varepsilon(t)\,dx + \frac{d}{dt} \int_\Omega G(u^\varepsilon(t))\,dx\\[1ex]
  = \int_\Omega \bigl( f(t) - f_\infty \bigr) \partial_t u^\varepsilon(t) \,dx + \frac{d}{dt} \int_\Omega f_\infty u^\varepsilon(t)\,dx, \forae t \in (0,\infty).\label{4.1.2}\noeqref{4.1.2}
\end{gather}
    The energy inequality \eqref{4.1.1} immediately follows from \eqref{4.1.2} by taking $s,t \in [0,\infty)$ with $s \leq t$ and integrating over $[s,t]$.

    Additionally, in view of the $L^\infty$-bound for $u^\varepsilon$ as in \eqref{Linfty}, we have
\begin{gather}
    \mathcal F_\varepsilon (u^\varepsilon(t)) \geq - \int_\Omega f_\infty u^\varepsilon(t)\,dx \geq - M \norm{ f_\infty }_{L^1(\Omega)} > -\infty, \forae t \in (0,\infty).\hspace{2em}\label{energyb}\noeqref{energyb}
\end{gather}
    The properties $\sharp$1)--$\sharp$3) are obtained as a consequence of the upper bound on $\mathcal F_\varepsilon$ derived from \eqref{4.1.1} and the lower bound on $\mathcal F_\varepsilon$ given in \eqref{energyb}.

    Thus, the proof of Lemma \ref{lem 4-1} is complete.
\end{proof}

\begin{proof}[Proof of Main Theorem 2]
  Let us divide the proof into the following two cases:

\begin{description}
    \item[\textmd{(Case I)}] There exists a sequence of times $ \{ \underline{t}^\varepsilon_n \}_{n \in \mathbb N} \subset (0,\infty)$ such that $\underline{t}^\varepsilon_n \uparrow \infty$ and $\bigl\{ \bigl| \sqrt{\beta(\underline{t}^\varepsilon_n)} \nabla u^\varepsilon(\underline{t}^\varepsilon_n) \bigr|_{[H]^N} \bigr\}_{n \in \mathbb N}$ is bounded;
    \item[\textmd{(Case I\hspace{-1pt}I)}] $\bigl\{ \bigl| \sqrt{\beta(t)} \nabla u^\varepsilon(t) \bigr|_{[H]^N} \bigr\}_{t \geq0}$ is unbounded.
    \end{description}
Note that these two cases cover all possibilities.
\medskip\\

    \noindent\underline{\textit{The proof in (Case I).}}\ \ From $\sharp$3) in Lemma \ref{lem 4-1}, it is immediately seen that
\begin{itemize}
  \item[$\flat$4)] $\{ u^\varepsilon(\underline{t}^\varepsilon_n) \}_{n\in\mathbb N} \text{ is bounded in } W^{1,1}(\Omega) \cap L^\infty(\Omega).$
\end{itemize}
    Since the embedding $BV(\Omega) \cap L^\infty(\Omega) \subset H$ is compact, we can find a subsequence $\{ \underline{t}^\varepsilon_n \}_{n\in\mathbb N} \subset (0,\infty)$ (not relabeled) and a function $\underline{u}^\varepsilon_\infty \in BV(\Omega) \cap L^\infty(\Omega)$ such that
\begin{gather*}
    \underline{t}^\varepsilon_1 < \underline{t}^\varepsilon_2 < \cdots < \underline{t}^\varepsilon_n \uparrow \infty \as n \to \infty,
\end{gather*}
and
\begin{gather}
  u^\varepsilon(\underline{t}^\varepsilon_n) \to \underline{u}^\varepsilon_\infty
  \begin{cases}
      & \In H,\\[1ex]
      \weaklystar & \In L^\infty(\Omega),\\[1ex]
      \weaklystar & \In BV(\Omega),
  \end{cases}
  \ \ \as n \to \infty.\label{4.1.4}\noeqref{4.1.4}
\end{gather}
Here, we define:
\begin{gather*}
  \underline{u}^\varepsilon_n(t) \coloneqq u^\varepsilon(t + \underline{t}^\varepsilon_n),\ \underline{\alpha}^\varepsilon_n(t) \coloneqq \alpha(t + \underline{t}^\varepsilon_n),\ \underline{\beta}^\varepsilon_n(t) \coloneqq \beta(t + \underline{t}^\varepsilon_n),\ \underline{f}^\varepsilon_n(t) \coloneqq f(t + \underline{t}^\varepsilon_n) \In H,\\[1ex] \text{for any } t \in[0,1] \text{ and } n \in \mathbb N.
\end{gather*}
Then, it follows from $\sharp$1) and $\sharp$2) in Lemma \ref{lem 4-1} that
\begin{gather}
  \begin{cases}
    \partial_t \underline{u}^\varepsilon_n\ (\text{ and } \sqrt{\underline{\alpha}^\varepsilon_n}\partial_t \underline{u}^\varepsilon_n) \to 0 \In \mathscr H_1,\\[1ex]
    \sqrt{\underline{\beta}^\varepsilon_n} \nabla \partial_t \underline{u}^\varepsilon_n \to 0 \In [\mathscr H_1]^N,
   \end{cases}
\as n \to \infty.\label{4.1.5}\noeqref{4.1.5}
\end{gather}
Additionally, one can observe from $\sharp$1) and $\sharp$3) that
\begin{itemize}
        \item[$\flat$5)] $ \{ \underline{u}^\varepsilon_n \}_{n \in \mathbb N}\text{ is bounded} \In L^\infty(0,1;W^{1,1}(\Omega)) \cap L^\infty(Q_1)$;
        \item[$\flat$6)] $\{ \partial_t \underline{u}^\varepsilon_n \}_{n \in \mathbb N}\text{ is bounded} \In \mathscr H_1$.
\end{itemize}
By the Aubin-type compactness theorem, we can find a function $\underline{u}^\varepsilon \in C([0,1];H)$ such that the following convergence holds upon passing to a subsequence:
\begin{gather}
  \underline{u}^\varepsilon_n \to \underline{u}^\varepsilon \In C([0,1];H) \as n \to \infty.\label{4.1.6}\noeqref{4.1.6}
\end{gather} 
Combining $\sharp$1) with \eqref{4.1.6}, we verify that
\begin{gather}
  \underline{u}^\varepsilon (t) = \underline{u}^\varepsilon_\infty \In H, \text{ for any } t \in [0,1],\label{4.1.7}\noeqref{4.1.7}
\end{gather}
via
\begin{gather*}
  \norm{ \underline{u}^\varepsilon_n(t) - \underline{u}^\varepsilon_\infty }_H = \norm{u^\varepsilon (t + \underline{t}^\varepsilon_n) - \underline{u}^\varepsilon_\infty }_H \leq \norm{ u^\varepsilon(t + \underline{t}^\varepsilon_n) - u^\varepsilon(\underline{t}^\varepsilon_n) }_H + \norm{ u^\varepsilon(\underline{t}^\varepsilon_n) - \underline{u}^\varepsilon_\infty }_H\\[1ex]
  \leq \int_{\underline{t}^\varepsilon_n}^\infty \norm{\partial_t u^\varepsilon(\sigma)}_H\,d\sigma + \norm{ u^\varepsilon(\underline{t}^\varepsilon_n) - \underline{u}^\varepsilon_\infty }_H \to 0 \as n \to \infty, \text{ for any } t \in [0,1].
\end{gather*}

Now, we show that $\underline{u}^\varepsilon_\infty \in S^\varepsilon_\infty$, i.e., $\underline{u}^\varepsilon_\infty$ is a solution to the steady-state problem (S)$_\varepsilon$. From (S1)$_\varepsilon$, we have
\begin{gather}
  \bigl( \underline{\alpha}^\varepsilon_n \partial_t \underline{u}^\varepsilon_n, \underline{u}^\varepsilon_n - \varphi \bigr)_{\mathscr H_1} + \bigl( \underline{\beta}^\varepsilon_n \nabla \partial_t \underline{u}^\varepsilon_n, \nabla( \underline{u}^\varepsilon_n - \varphi ) \bigr)_{[\mathscr H_1]^N}\\[1ex]
  + \int_0^1 \int_\Omega \gamma_\varepsilon ( \nabla \underline{u}^\varepsilon_n (t) )\,dxdt - \int_0^1 \int_\Omega \gamma_\varepsilon ( \nabla \varphi )\,dxdt \leq \bigl( \underline{f}^\varepsilon_n - g(\underline{u}^\varepsilon_n), \underline{u}^\varepsilon_n - \varphi \bigr)_{\mathscr H_1},\\[1ex]
  \text{ for any } \varphi \in V \text{ and } n \in \mathbb N.\label{4.1.8}\noeqref{4.1.8}
\end{gather}
Here, since the assumption (A6) implies
\begin{gather}
    \norm{\underline{f}^\varepsilon_n - f_\infty}_{\mathscr H_1}^2 = \int_0^1 \int_\Omega | \underline{f}^\varepsilon_n (t) - f_\infty |^2\,dxdt \to 0 \as n\to\infty,\label{4.1.9}\noeqref{4.1.9}
\end{gather}
having in mind \eqref{4.1.5}--\eqref{4.1.7} and \eqref{4.1.9}, and letting $n \to \infty$ in \eqref{4.1.8}, one can derive
\begin{align}
    & 0 +\liminf_{n \to \infty} \bigl( \underline{\beta}^\varepsilon_n \nabla \partial_t \underline{u}^\varepsilon_n, \nabla( \underline{u}^\varepsilon_n - \varphi ) \bigr)_{[\mathscr H_1]^N}
    \\[1ex]
    & \qquad + \int_\Omega \gamma_\varepsilon ( D \underline{u}^\varepsilon_\infty ) - \int_\Omega \gamma_\varepsilon ( \nabla \varphi )\,dx \leq \bigl( f_\infty - g(\underline{u}^\varepsilon_\infty ), \underline{u}^\varepsilon_\infty - \varphi \bigr)_H, 
    \\
    & \text{for any } \varphi \in V.\label{4.1..10}\noeqref{4.1..10}
\end{align}
Since $V$ is dense in $BV(\Omega) \cap H$, it remains to show that the $\liminf$-term in \eqref{4.1..10} either is equal to 0 or is nonnegative. In case (I), we will prove that
\begin{gather}
    \bigl( \underline{\beta}^\varepsilon_n \nabla \partial_t \underline{u}^\varepsilon_n , \nabla( \underline{u}^\varepsilon_n - \varphi ) \bigr)_{ [\mathscr H_1]^N } \to 0 \as n \to \infty.\label{cc}\noeqref{cc}
\end{gather}
To this end, it suffices to show the following boundedness:
\begin{itemize}
 \item[$\flat$7)] $\bigl\{ \sqrt{ \underline{\beta}^\varepsilon_n } \nabla \underline{u}^\varepsilon_n \bigr\}_{n \in \mathbb N} \text{ is bounded in } [ \mathscr H_1 ]^N$.
\end{itemize}
Using \eqref{4.1.5}--\eqref{4.1.9} together with the boundedness of the convergent sequences, we see that
\begin{gather}
  \int_0^t \bigl( \underline{\beta}^\varepsilon_n (\sigma) \nabla \partial_t \underline{u}^\varepsilon_n (\sigma), \nabla \underline{u}^\varepsilon_n (\sigma) \bigr)_{[H]^N}\,d\sigma \leq \widetilde{C}_2, \text{ for any } t \in [0,1] \text{ and } n \in \mathbb N,
\end{gather}
where
\begin{gather*}
  \widetilde{C}_2 \coloneqq \sup_{n \in \mathbb N} \left\{ \bigl( \underline{\alpha}^\varepsilon_n \partial_t \underline{u}^\varepsilon_n + g(\underline{u}^\varepsilon_n) + \underline{f}^\varepsilon_n, \underline{u}^\varepsilon_n - \varphi \bigr)_{\mathscr H_1} + \bigl( \underline{\beta}^\varepsilon_n \nabla \partial_t \underline{u}^\varepsilon_n, \nabla \varphi \bigr)_{[\mathscr H_1]^N} \right\} + \norm{ \gamma_\varepsilon (\nabla \varphi) }_{L^1(\Omega)}.
\end{gather*}
On the other hand, the decay condition on $\beta$ in (A2) allows us to derive the following:
\begin{align*}
  &\int_0^t \bigl( \underline{\beta}^\varepsilon_n (\sigma) \nabla \partial_t \underline{u}^\varepsilon_n (\sigma), \nabla \underline{u}^\varepsilon_n (\sigma) \bigr)_{[H]^N}\\[1ex]
  & = \frac{1}{2} \norm{ \sqrt{ \underline{\beta}^\varepsilon_n (t) } \nabla \underline{u}^\varepsilon_n (t) }_{[H]^N}^2 - \frac{1}{2} \norm{ \sqrt{ \underline{\beta}^\varepsilon_n (0) } \nabla \underline{u}^\varepsilon_n (0) }_{[H]^N}^2 - \frac{1}{2} \int_0^t \partial_t \underline{\beta}^\varepsilon_n (\sigma) | \nabla \underline{u}^\varepsilon_n (\sigma) |^2\,dxd\sigma\\[1ex]
  & \geq \frac{1}{2} \norm{ \sqrt{ \underline{\beta}^\varepsilon_n (t) } \nabla \underline{u}^\varepsilon_n (t) }_{[H]^N}^2 - \frac{1}{2} \norm{ \sqrt{ \beta (\underline{t}^\varepsilon_n) } \nabla u^\varepsilon (\underline{t}^\varepsilon_n) }_{[H]^N}^2, \text{ for any } t\in[0,1] \text{ and } n \in \mathbb N,
\end{align*}
and hence
\begin{gather*}
  \norm{ \sqrt{ \underline{\beta}^\varepsilon_n (t) } \nabla \underline{u}^\varepsilon_n (t) }_{[H]^N}^2 \leq 2 \widetilde{C}_2 + \norm{ \sqrt{ \beta(\underline{t}^\varepsilon_n) } \nabla u^\varepsilon (\underline{t}^\varepsilon_n) }_{[H]^N}^2, \text{ for any } t \in [0,1] \text{ and } n \in \mathbb N.
\end{gather*}
This inequality and the assumption (I) imply that the boundedness as in $\flat$7) holds.

Thus, we conclude that $\underline{u}^\varepsilon_\infty \in S^\varepsilon_\infty$.
\medskip\\

\noindent\underline{\textit{The proof in (Case I\hspace{-1pt}I).}}\ \ We apply Lemma \ref{lem 5-1} with $a(t) \coloneqq |\sqrt{\beta(t)}\nabla u^\varepsilon(t)|_{[H]^N}$. Then, we can find a sequence of times $\{ \overline{t}^\varepsilon_n \}_{n\in\mathbb N} \subset (0,\infty)$ with $\overline{t}^\varepsilon_n \uparrow \infty$ such that
\begin{gather}
  \norm{ \sqrt{ \vphantom{\rule{0pt}{2ex}}\smash{\beta(\overline{t}^\varepsilon_n)} } \nabla u^\varepsilon (\overline{t}^\varepsilon_n) }_{[H]^N} \leq \norm{ \sqrt{ \vphantom{\rule{0pt}{2ex}}\smash{\beta(\overline{t}^\varepsilon_n + 1)} } \nabla u^\varepsilon (\overline{t}^\varepsilon_n + 1) }_{[H]^N}, \text{ for any }n \in \mathbb N.\label{4.1..12}\noeqref{4.1..12}
\end{gather}
Proceeding as in (Case I), we can find a subsequence $\{ \overline{ t}^\varepsilon_n  \}_{n\in\mathbb N} \subset (0,\infty)$ (not relabeled) and a function $\overline{u}^\varepsilon_\infty \in BV(\Omega) \cap L^\infty(\Omega)$ such that
\begin{gather*}
  u^\varepsilon(\overline{t}^\varepsilon_n) \to \overline{u}^\varepsilon_\infty
  \begin{cases}
      & \In H,\\[1ex]
      \weaklystar & \In L^\infty(\Omega),\\[1ex]
      \weaklystar & \In BV(\Omega),
  \end{cases}
  \ \ \as n \to \infty,
\end{gather*}
and
\begin{gather}
  0 + \liminf_{n \to \infty} \bigl( \overline{\beta}^\varepsilon_n \nabla \partial_t \overline{u}^\varepsilon_n, \nabla( \overline{u}^\varepsilon_n - \varphi ) \bigr)_{[\mathscr H_1]^N}\\[1ex]
  + \int_\Omega \gamma_\varepsilon ( D \overline{u}^\varepsilon_\infty ) - \int_\Omega \gamma_\varepsilon ( \nabla \varphi )\,dx \leq \bigl( f_\infty - g(\overline{u}^\varepsilon_\infty ), \overline{u}^\varepsilon_\infty - \varphi \bigr)_H, \text{ for any } \varphi \in V,\label{4.1..13}\noeqref{4.1..13}
\end{gather}
where
\begin{gather*}
    \overline{u}^\varepsilon_n (t) \coloneqq u^\varepsilon (t + \overline{t}^\varepsilon_n) \text{ and } \overline{\beta}^\varepsilon_n (t) \coloneqq \beta(t + \overline{t}^\varepsilon_n) \In H,\\[1ex] \text{for any } t \in[0,1] \text{ and } n \in \mathbb N.
\end{gather*}
In (Case I\hspace{-1pt}I), we establish the nonnegativity of the $\liminf$-term. The assumptions \eqref{4.1..12} and (A2) with \eqref{4.1.5} allow us to derive that
\begin{align}
  &\bigl( \overline{\beta}^\varepsilon_n \nabla \partial_t \overline{u}^\varepsilon_n , \nabla( \overline{u}^\varepsilon_n - \varphi ) \bigr)_{ [\mathscr H_1]^N }\\[1ex]
  &\hspace{1em} = \bigl( \sqrt{ \vphantom{\rule{0pt}{2ex}}\smash{\overline{\beta}^\varepsilon_n} } \nabla \partial_t \overline{u}^\varepsilon_n , \sqrt{ \vphantom{\rule{0pt}{2ex}}\smash{\overline{\beta}^\varepsilon_n} } \nabla \overline{u}^\varepsilon_n \bigr)_{ [\mathscr H_1]^N } - \bigl( \sqrt{ \vphantom{\rule{0pt}{2ex}}\smash{\overline{\beta}^\varepsilon_n} } \nabla \partial_t \overline{u}^\varepsilon_n, \sqrt{ \vphantom{\rule{0pt}{2ex}}\smash{\overline{\beta}^\varepsilon_n} } \nabla \varphi \bigr)_{ [\mathscr H_1]^N }\\[1ex]
  &\hspace{1em} \geq - \bigl( \sqrt{ \vphantom{\rule{0pt}{2ex}}\smash{\overline{\beta}^\varepsilon_n} } \nabla \partial_t \overline{u}^\varepsilon_n, \sqrt{ \vphantom{\rule{0pt}{2ex}}\smash{\overline{\beta}^\varepsilon_n} } \nabla \varphi \bigr)_{ [\mathscr H_1]^N } \to 0 \as n \to \infty,\label{4.1..14}\noeqref{4.1..14}
\end{align}
via
\begin{align*}
  &\bigl( \sqrt{ \vphantom{\rule{0pt}{2ex}}\smash{\overline{\beta}^\varepsilon_n} } \nabla \partial_t \overline{u}^\varepsilon_n , \sqrt{ \vphantom{\rule{0pt}{2ex}}\smash{\overline{\beta}^\varepsilon_n} } \nabla \overline{u}^\varepsilon_n \bigr)_{ [\mathscr H_1]^N }\\[1ex]
  &\hspace{1em} = \frac{1}{2} \int_0^1 \frac{d}{dt} \norm{ \sqrt{ \vphantom{\rule{0pt}{2ex}}\smash{ \overline{ \beta}^\varepsilon_n (t) } } \nabla \overline{u}^\varepsilon_n (t) }_{[H]^N}^2\,dt - \frac{1}{2} \int_0^1 \int_\Omega \partial_t \overline{\beta}^\varepsilon_n (t) | \nabla \overline{u}^\varepsilon_n (t) |^2\,dxdt\\[1ex]
  &\hspace{1em} \geq \frac{1}{2} \left\{ \norm{ \sqrt{ \vphantom{\rule{0pt}{2ex}} \smash{ \beta( \overline{t}^\varepsilon_n + 1 ) } } \nabla u^\varepsilon ( \overline{t}^\varepsilon_n + 1 ) }_{[H]^N}^2 - \norm{ \sqrt{ \vphantom{\rule{0pt}{2ex}} \smash{ \beta( \overline{t}^\varepsilon_n ) } } \nabla u^\varepsilon ( \overline{t}^\varepsilon_n ) }_{[H]^N}^2 \right\} \geq 0, \text{ for any }n\in\mathbb N.
\end{align*}
\eqref{4.1..13} and \eqref{4.1..14} imply that $\overline{u}^\varepsilon_\infty \in S^\varepsilon_\infty$.

Thus, combining Cases (I) and \II completes the proof of Main Theorem 2.
\end{proof}

\subsection{Proof of Main Theorem 3}
As seen in the proof of Main Theorem 2, whether $|\sqrt{\beta}(\cdot)\nabla u^\varepsilon(\cdot)|_{[H]^N}$ remains bounded in time is a key point in the large-time analysis. In this light, we first prove the following lemma. 
\begin{lemma}\label{lem 4-2}
    Under the assumptions (A1)--(A6) and (A2)$^\prime$, let $\varepsilon \in [0,1)$, and let $u^\varepsilon$ be a unique solution to the problem (P)$_\varepsilon$. Then, it holds that:
\begin{itemize}
  \item[$\sharp$4)] $ \sqrt{\beta} \nabla u^\varepsilon \in L^\infty(0,\infty;[H]^N) $.
\end{itemize}
\end{lemma}
\begin{proof}
  Let us take $\varphi \coloneqq u^\varepsilon(t)$ as a test function in (S1)$_\varepsilon$. Then, we have
\begin{gather}
      \bigl( \alpha(t) \partial_t u^\varepsilon(t) + g( u^\varepsilon(t) ), u^\varepsilon(t) \bigr)_H + \bigl( \bm{\omega}^*_\varepsilon(t) + \beta(t) \nabla \partial_t u^\varepsilon(t), \nabla u^\varepsilon(t)  \bigr)_{[H]^N}\\[1ex]
      = \bigl( f(t), u^\varepsilon(t) \bigr)_H, \text{ for} \ae t \in (0,\infty).\label{4.2.1}\noeqref{4.2.1}
\end{gather}
        Here, from $\sharp$3) in Lemma \ref{lem 4-1}, the monotonicity of $\partial \gamma_\varepsilon$, and the exponential decay condition on $\beta$ in (A2)$^\prime$, we derive the estimate
\begin{gather}
  \frac{d}{dt} \left( \norm{ \sqrt{ \alpha(t) } u^\varepsilon(t) }_H^2 + \norm{ \sqrt{\beta(t)} \nabla u^\varepsilon(t) }_{[H]^N}^2 \right) + C \norm{ \sqrt{ \beta(t) } \nabla u^\varepsilon(t) }_{[H]^N}^2\\[1ex]
  \leq \widetilde{C}_3 ( \norm{ \partial_t \alpha(t) }_{L^1(\Omega)} + 1 ) ( \norm{u^\varepsilon}_{L^\infty(Q)}^2 + \norm{f}_{L^\infty(Q)}^2 ), \forae t \in (0,\infty),\label{4.2.2}\noeqref{4.2.2}
\end{gather}
via
\begin{align*}
  \bigl( \alpha(t) \partial_t u^\varepsilon(t), u^\varepsilon(t) \bigr)_H &= \frac{1}{2} \frac{d}{dt} \norm{ \sqrt{ \alpha(t) } u^\varepsilon(t) }_H^2 - \frac{1}{2} \int_\Omega \partial_t \alpha(t) | u^\varepsilon (t) |^2\,dx\\[1ex]
  & \geq \frac{1}{2} \frac{d}{dt} \norm{ \sqrt{ \alpha(t) } u^\varepsilon(t) }_H^2 - \frac{1}{2} \norm{ u^\varepsilon }_{L^\infty(Q)}^2 \norm{ \partial_t \alpha(t) }_{L^1(\Omega)},
\end{align*}
\begin{align*}
  \bigl( \beta(t) \nabla \partial_t u^\varepsilon(t), \nabla u^\varepsilon(t) \bigr)_{[H]^N} &= \frac{1}{2} \frac{d}{dt} \norm{ \sqrt{\beta(t)} \nabla u^\varepsilon (t) }_{[H]^N}^2 - \frac{1}{2} \int_\Omega \partial_t \beta(t) | \nabla u^\varepsilon (t) |^2\,dx\\[1ex]
  &\geq \frac{1}{2} \frac{d}{dt} \norm{ \sqrt{\beta(t)} \nabla u^\varepsilon (t) }_{[H]^N}^2 + \frac{1}{2} C \norm{ \sqrt{\beta(t)} \nabla u^\varepsilon (t) }_{[H]^N}^2,
\end{align*}
\begin{gather*}
  \bigl( \bm{\omega}^*_\varepsilon (t), \nabla u^\varepsilon (t) \bigr)_{[H]^N} \geq 0,
\end{gather*}
\begin{gather*}
  - \bigl( g(u^\varepsilon (t)), u^\varepsilon(t) \bigr)_H \leq \left( \norm{g^\prime}_{L^\infty(-M,M)} + \frac{1}{2}|g(0)|^2 \right) \norm{ u^\varepsilon }_{L^\infty(Q)}^2 \mathcal L^N(\Omega),
\end{gather*}
and
\begin{gather*}
  \bigl( f(t), u^\varepsilon(t) \bigr)_H \leq \frac{1}{2} ( \norm{u^\varepsilon}_{L^\infty(Q)}^2 + \norm{f}_{L^\infty(Q)}^2 ) \mathcal L^N (\Omega), \forae t \in (0,\infty),
\end{gather*}
where
\begin{gather*}
  \widetilde{C}_3 \coloneqq 2 ( 2 \norm{g^\prime}_{L^\infty(-M,M)} + |g(0)|^2 + 1 ) ( \mathcal L^N (\Omega) + 1 ). 
\end{gather*}
Additionally, a direct calculation gives the following estimate:
\begin{gather}
  C \norm{ \sqrt{\alpha(t)} u^\varepsilon (t) }_H^2 \leq C \norm{\alpha}_{L^\infty(Q)} \norm{u^\varepsilon}_{L^\infty(Q)}^2 \mathcal L^N (\Omega), \ae t \in (0,T).\label{4.2.3}\noeqref{4.2.3}
\end{gather}
        Adding \eqref{4.2.2} and \eqref{4.2.3} yields
\begin{gather}
  \frac{d}{dt} Y^\varepsilon(t) + C Y^\varepsilon(t) \leq \widetilde{C}_4 ( \norm{ \partial_t \alpha(t) }_{L^1(\Omega)} + 1 ) ( \norm{u^\varepsilon}_{L^\infty(Q)}^2 + \norm{f}_{L^\infty(Q)}^2 ),\\
 \forae t \in (0,\infty),\label{GronY}\noeqref{GronY}
\end{gather}
where
\begin{gather*}
  Y^\varepsilon(t) \coloneqq \norm{ \sqrt{ \alpha(t) } u^\varepsilon (t) }_H^2 + \norm{ \sqrt{\beta(t)} \nabla u^\varepsilon (t) }_{[H]^N}^2, \text{ for any } t \in [0,\infty),
\end{gather*}
and
\begin{gather*}
  \widetilde{C}_4 \coloneqq \widetilde{C}_3 + C \norm{\alpha}_{L^\infty(Q)} \mathcal L^N (\Omega).
\end{gather*}

    Now, applying Gronwall's lemma to \eqref{GronY}, we obtain that
\begin{align*}
  Y^\varepsilon&(t) \leq e^{-Ct} Y^\varepsilon (0) + \widetilde{C}_4 (\norm{u^\varepsilon}_{L^\infty(Q)}^2 + \norm{f}_{L^\infty(Q)}^2) \int_0^t ( \norm{ \partial_t \alpha(s) }_{L^1(\Omega)} + 1 ) e^{C(s-t)}\,ds\\[1ex]
  &\hspace{-1.4em}\leq Y^\varepsilon(0) + \widetilde{C}_4 (\norm{u^\varepsilon}_{L^\infty(Q)}^2 + \norm{f}_{L^\infty(Q)}^2) \int_0^t \left( \frac{1}{2} \norm{ \partial_t \alpha(s) }_{L^1(\Omega)}^2 + \frac{1}{2} e^{2C(s-t)} + e^{C(s-t)} \right)\,ds\\[1ex]
  &\hspace{-1.4em}\leq Y^\varepsilon(0) + \widetilde{C}_4 (\norm{u^\varepsilon}_{L^\infty(Q)}^2 + \norm{f}_{L^\infty(Q)}^2) \left( \frac{1}{2} \norm{ \partial_t \alpha }_{L^2(0,\infty;L^1(\Omega))}^2 + \frac{5}{4C} \right) \hspace{-0.2em}<\hspace{-0.2em} \infty, \text{ for any }t \in [0,\infty).
\end{align*}

  This implies that $Y^\varepsilon \in L^\infty (0,\infty)$, and hence Lemma \ref{lem 4-2} follows.
\end{proof}

\begin{proof}[Proof of Main Theorem 3]
    Let us take any $\omega$-limit point $ u_\infty^\varepsilon \in \omega(u) $. Then, from Definition \ref{def 4} and $\sharp$3) in Lemma \ref{lem 4-1}, 
    we can find a sequence of times $\{ t^\varepsilon_n \}_{n \in \mathbb N} \subset (0,\infty)$ with $t^\varepsilon_n \uparrow \infty$ such that for the unique solution $u^\varepsilon$ to the problem (P)$_\varepsilon$,
\begin{gather}
    u^\varepsilon (t^\varepsilon_n) \to u^\varepsilon_\infty
    \begin{cases}
      & \In H,\\[1ex]
      \weaklystar & \In L^\infty(\Omega),\\[1ex]
      \weaklystar & \In BV(\Omega),
    \end{cases}
    \ \ \as n\to\infty.\label{4.2.4}\noeqref{4.2.4}
  \end{gather}
Here, we define:
\begin{gather*}
  u^\varepsilon_n (t) \coloneqq u^\varepsilon (t + t^\varepsilon_n),\ \alpha^\varepsilon_n (t) \coloneqq \alpha(t + t^\varepsilon_n),\ \beta^\varepsilon_n (t) \coloneqq \beta(t + t^\varepsilon_n),\ f^\varepsilon_n (t) \coloneqq f(t + t^\varepsilon_n) \In H,\\[1ex]
\text{ for any } t \in[0,1] \text{ and }  n \in \mathbb N.
\end{gather*}
Then, it follows from $\sharp$1) and $\sharp$2) in Lemma \ref{lem 4-1} that
\begin{gather}
  \begin{cases}
    \partial_t u^\varepsilon_n \ (\text{ and } \sqrt{ \vphantom{\rule{0pt}{2ex}} \smash{ \alpha^\varepsilon_n} } \partial_t u^\varepsilon_n) \to 0 \In \mathscr H_1,\\[1ex]
    \sqrt{ \vphantom{\rule{0pt}{2.5ex}} \smash{ \beta^\varepsilon_n} }  \nabla \partial_t u^\varepsilon_n \to 0 \In [\mathscr H_1]^N,
   \end{cases}
\as n \to \infty.\label{4.2.5}\noeqref{4.2.5}
\end{gather}
Also, the Aubin-type compactness theorem with $\sharp$1) and $\sharp$3) enables us to find a subsequence $\{t^\varepsilon_n\}_{n \in \mathbb N} \subset (0,\infty)$ with $t^\varepsilon_n \uparrow \infty$ (not relabeled) and a function $\widetilde{u}^\varepsilon \in C([0,1];H)$ such that
\begin{gather}
  u^\varepsilon_n \to \widetilde{u}^\varepsilon \In C([0,1];H) \as n \to \infty,\label{4.2.6}\noeqref{4.2.6}
\end{gather} 
and by the same argument as in the proof of \eqref{4.1.7}, 
\begin{gather}
  \widetilde{u}^\varepsilon (t) = u^\varepsilon_\infty \In H, \text{ for any } t \in [0,1]. \label{4.2.7}\noeqref{4.2.7}
\end{gather}

    Of particular importance is the following observation. For any $\varphi \in V$, the strong convergence \eqref{4.2.5} and the boundedness $\sharp$4) in Lemma \ref{lem 4-2}, together with (A2), yield that
    \begin{align}
        \bigl( \beta^\varepsilon_n & \nabla \partial_t u^\varepsilon_n, \nabla ( u^\varepsilon_n - \varphi) \bigr)_{[\mathscr H_1]^N}
        \\
         & = \bigl( \sqrt{\beta^\varepsilon_n} \nabla \partial_t u^\varepsilon_n, \sqrt{\beta^\varepsilon_n}\nabla u^\varepsilon_n \bigr)_{[\mathscr H_1]^N} - \bigl( \sqrt{\beta^\varepsilon_n} \nabla \partial_t u^\varepsilon_n, \sqrt{\beta^\varepsilon_n}\nabla \varphi \bigr)_{[\mathscr H_1]^N}
         \\
         & 
         \to 0, \mbox{ as $ n \to \infty $.}\label{3beta}\noeqref{3beta}
    \end{align}

Now, applying \eqref{4.2.5}--\eqref{3beta}, and proceeding as in the proof of Main Theorem 2, we obtain the following:
\begin{gather*}
  \int_\Omega \gamma_\varepsilon ( D u^\varepsilon_\infty ) - \int_\Omega \gamma_\varepsilon ( \nabla \varphi )\,dx \leq \bigl( f_\infty - g(u^\varepsilon_\infty), u^\varepsilon_\infty - \varphi \bigr)_H, \text{ for any } \varphi \in V.
\end{gather*}

This implies that $u^\varepsilon_\infty \in S^\varepsilon_\infty$, and hence we finish the proof of Main Theorem 3.
\end{proof}

\subsection{Proof of Main Theorem 4}
Main Theorem 4 is established by building on the proof of Main Theorem 3. The key points are the time-rescaling and the integrability property to be established, which are given, respectively, by
\begin{gather}
  w^\varepsilon (s) \coloneqq u^\varepsilon (e^s - 1) \In H, \text{ for any } s \in [0,\infty) \text{ with } t \coloneqq e^s - 1,\label{4.3.1}\noeqref{4.3.1}\\[1ex]
  \partial_s w^\varepsilon = e^s [\partial_t u^\varepsilon](e^s -1) \in \mathscr H, \text{ for any } \varepsilon \in [0,1).\label{4.3.2}\noeqref{4.3.2}
\end{gather}

First of all, for any $\varepsilon \in [0,1)$, under the time-rescaling \eqref{4.3.1}, the problem (P)$_\varepsilon$ can be rewritten as follows:
\begin{gather*}
    \text{(PS)}_\varepsilon
    \begin{cases}
      e^{-s} \widetilde{\alpha} \partial_s w^\varepsilon - \diver \left( \partial \gamma_\varepsilon( \nabla w^\varepsilon ) + e^{-s} \widetilde{\beta} \nabla \partial_s w^\varepsilon \right) + g(w^\varepsilon) \ni \widetilde{f} \text{ in } H, \forae s \in (0,\infty), \\[1.5ex]
      \left( \partial \gamma_\varepsilon( \nabla w^\varepsilon ) + e^{-s} \widetilde{\beta} \nabla \partial_s w^\varepsilon \right) \cdot n_\Gamma \ni 0 \text{ for}\aeon \Gamma \text{ and} \ae s \in (0,\infty),\\[1.5ex]
      w^\varepsilon(0,x) = u_0(x), \forae x \in \Omega,
    \end{cases}
  \end{gather*}
where
\begin{gather}
  \widetilde{\alpha}(s) \coloneqq \alpha(e^s - 1),\ \widetilde{\beta}(s) \coloneqq \beta(e^s - 1),\ \widetilde{f}(s) \coloneqq f(e^s - 1)
    \label{4.3.3}\noeqref{4.3.3}
    \\[1ex]
    \In H, \text{ for any } s \in [0,\infty).
\end{gather}

For the problem (P)$_\varepsilon$, the energy inequality \eqref{4.1.1} directly yields $\partial_t u^\varepsilon \in \mathscr H$. In contrast, for (PS)$_\varepsilon$, the same argument gives only the weighted estimate $e^{-s/2}\partial_s w^\varepsilon \in \mathscr H$. Therefore, to verify \eqref{4.3.2}, we first consider the linearized problem associated with (PS)$_\varepsilon$ and prove the following lemma.

\begin{lemma}\label{lem 4-3}
    Assume (A1)--(A7), and let $\varepsilon \in (0,1)$. Let $u^\varepsilon$ be the unique solution to (P)$_\varepsilon$. With the time-rescaling \eqref{4.3.1} and the associated functions defined in \eqref{4.3.3}, we consider the following linear pseudo-parabolic problem:
\begin{gather*}
  \rm{(LPS)}_\varepsilon
  \begin{cases}
    e^{-s} \widetilde{\alpha} \partial_s v^\varepsilon
    + \partial_s \bigl( e^{-s} \widetilde{\alpha} \bigr)v^\varepsilon
    - \diver \left(
      \nabla^2 \gamma_\varepsilon (\nabla w^\varepsilon) \nabla v^\varepsilon
      + e^{-s}\widetilde{\beta} \nabla \partial_s v^\varepsilon
      + \partial_s \bigl( e^{-s} \widetilde{\beta} \bigr) \nabla v^\varepsilon
    \right)\\[1.5ex]
    \hspace{13em}
    + g^\prime(w^\varepsilon) v^\varepsilon
    = \partial_s \widetilde{f} \In H,
    \forae s \in (0,\infty),\\[1.5ex]
    \left(
      e^{-s}\widetilde{\beta} \nabla \partial_s v^\varepsilon
      + \partial_s \bigl( e^{-s} \widetilde{\beta} \bigr) \nabla v^\varepsilon
    \right) \cdot n_\Gamma
    = 0
    \forae \text{on } \Gamma \text{ and} \ae s \in (0,\infty),\\[1.5ex]
    v^\varepsilon (0)
    = \bigl(
      \alpha(0) I - \diver ( \beta(0) \nabla )
    \bigr)^{-1}
    \bigl(
      f(0) - g(u_0)
      + \diver ( \nabla \gamma_\varepsilon ( \nabla u_0 ) )
    \bigr)
    \in V.
  \end{cases}
\end{gather*}
Then, the following assertions hold:
\begin{itemize}
  \item[(i)] (LPS)$_\varepsilon$ admits a unique solution $ v^\varepsilon \in W^{1,2}_\loc([0,\infty);V) $;
  \item[(ii)] $\partial_s w^\varepsilon = v^\varepsilon \aein Q$.
\end{itemize}
\end{lemma}
\begin{proof}
  Let $\varepsilon \in (0,1)$ be fixed. Then, the assertion (i) follows directly from the argument used in the proof of Main Theorem 1. 

    We now verify the assertion (ii). Let $v^\varepsilon \in W^{1,2}_\loc([0,\infty);V)$ be the unique solution to (LPS)$_\varepsilon$. Then, for any $S \in (0,\infty)$, $s \in [0, S]$, and $\varphi \in V$, integration over $(0,s)$ gives
\begin{align}
  &\int_0^s \bigl( e^{-\sigma} \widetilde{\alpha}(\sigma) \partial_s v^\varepsilon(\sigma) + \partial_s \bigl( e^{-\sigma} \widetilde{\alpha}(\sigma) \bigr) v^\varepsilon(\sigma), \varphi \bigr)_H\, d\sigma \\[1ex]
  &\hspace{5em}+ \int_0^s \bigl( e^{-\sigma} \widetilde{\beta}(\sigma)  \nabla \partial_s v^\varepsilon(\sigma) + \partial_s \bigl( e^{-\sigma} \widetilde{\beta}(\sigma) \bigr) \nabla v^\varepsilon(\sigma), \nabla \varphi \bigr)_{[H]^N}\, d\sigma\\[1ex]
  & = \int_0^s \bigl( \partial_s \widetilde{f}(\sigma) - g^\prime(w^\varepsilon(\sigma)) \partial_s w^\varepsilon(\sigma), \varphi \bigr)_H\,d\sigma\\[1ex]
  &\hspace{5em} - \int_0^s \bigl( \nabla^2 \gamma_\varepsilon (\nabla w^\varepsilon(\sigma)) \nabla \partial_s w^\varepsilon(\sigma), \nabla \varphi \bigr)_{[H]^N}\,d\sigma, \text{ for any } s \in [0,S].\label{4.3.4}\noeqref{4.3.4}
\end{align}
    Also, the initial condition of (LPS)$_\varepsilon$ implies that:
\begin{gather}
  \hspace{-5em}\bigl( \alpha(0) v^\varepsilon(0), \varphi \bigr)_H + \big( \beta(0) \nabla v^\varepsilon(0), \nabla \varphi \bigr)_{[H]^N} 
    \\
    \hspace{5em}= \bigl( f(0) - g(u_0), \varphi \bigr)_H - \bigl( \nabla \gamma_\varepsilon (u_0), \nabla \varphi \bigr)_{[H]^N}.
\label{4.3.5}\noeqref{4.3.5}
\end{gather}
  Here, we define:
  \begin{gather*}
    \widetilde{w}^\varepsilon(s) \coloneqq u_0 + \int_0^s v^\varepsilon(\sigma)\,d\sigma, \text{ for any } s \in [0,S].
  \end{gather*}
    Then, combining \eqref{4.3.4} and \eqref{4.3.5}, we arrive at
\begin{align}
  &\bigl( e^{-s} \widetilde{\alpha}(s) \partial_s \widetilde{w}^\varepsilon(s), \varphi \bigr)_H + \bigl( e^{-s} \widetilde{\beta}(s) \nabla \partial_s \widetilde{w}^\varepsilon(s), \nabla \varphi \bigr)_{[H]^N}\\[1ex]
&\hspace{2em}= \bigl( \widetilde{f}(s) - g(w^\varepsilon(s)), \varphi \bigr)_H - \bigl( \nabla \gamma_\varepsilon ( \nabla w^\varepsilon(s) ), \nabla \varphi \bigr)_{[H]^N}, \forae s \in (0,S),\hspace{3em}\label{4.3.6}\noeqref{4.3.6}
\end{align}
    via the identities
\begin{gather*}
  \int_0^s \bigl( e^{-\sigma} \widetilde{\alpha}(\sigma) \partial_s v^\varepsilon(\sigma) + \partial_s \bigl( e^{-\sigma} \widetilde{\alpha}(\sigma) \bigr) v^\varepsilon(\sigma), \varphi \bigr)_H\, d\sigma = \bigl( e^{-s} \widetilde{\alpha}(s) v^\varepsilon(s) - \alpha(0)v^\varepsilon(0), \varphi \bigr)_H,
\end{gather*}
\begin{align*}
  &\int_0^s \bigl( e^{-\sigma} \widetilde{\beta}(\sigma) \nabla \partial_s v^\varepsilon (\sigma) + \partial_s \bigl( e^{-\sigma} \widetilde{\beta}(\sigma) \bigr) \nabla v^\varepsilon (\sigma), \nabla \varphi \bigr)_{[H]^N}\, d\sigma \\[1ex]
  &\hspace{17em}= \bigl( e^{-s} \widetilde{\beta}(s) \nabla v^\varepsilon(s) - \beta(0) \nabla v^\varepsilon (0), \nabla \varphi \bigr)_{[H]^N},
\end{align*}
\begin{gather*}
  \int_0^s \bigl( \partial_s \widetilde{f}(\sigma) - g^\prime(w^\varepsilon(\sigma)) \partial_s w^\varepsilon(\sigma), \varphi \bigr)_H\, d\sigma = \bigl( \widetilde{f}(s) - g(w^\varepsilon(s)) - f(0) + g(u_0), \varphi \bigr)_H,
\end{gather*}
and
\begin{gather*}
  \int_0^s \bigl( \nabla^2 \gamma_\varepsilon (\nabla w^\varepsilon(\sigma)) \nabla \partial_s w^\varepsilon(\sigma), \nabla \varphi \bigr)_{[H]^N}\,d\sigma = \bigl( \nabla \gamma_\varepsilon ( \nabla w^\varepsilon(s) ) - \nabla \gamma_\varepsilon ( \nabla u_0 ), \nabla \varphi \bigr)_{[H]^N},\\
 \text{ for any } s \in [0,S].
\end{gather*}

    On the other hand, the variational formulation of (PS)$_\varepsilon$ gives
\begin{align}
  &\bigl( e^{-s} \widetilde{\alpha}(s) \partial_s w^\varepsilon(s), \varphi \bigr)_H + \bigl( e^{-s} \widetilde{\beta}(s) \nabla \partial_s w^\varepsilon(s), \nabla \varphi \bigr)_{[H]^N}\\[1ex]
&\hspace{2em}= \bigl( \widetilde{f}(s) - g(w^\varepsilon(s)), \varphi \bigr)_H - \bigl( \nabla \gamma_\varepsilon ( \nabla w^\varepsilon(s) ), \nabla \varphi \bigr)_{[H]^N}, \forae s \in (0,S).\hspace{3em}\label{4.3.7}\noeqref{4.3.7}
\end{align}
Subtracting \eqref{4.3.7} from \eqref{4.3.6} and taking
$\varphi \coloneqq \partial_s (\widetilde{w}^\varepsilon-w^\varepsilon)(s)$, we obtain
\begin{align*}
  0 &= \int_\Omega e^{-s} \widetilde{\alpha}(s) | \partial_s( \widetilde{w}^\varepsilon - w^\varepsilon )(s) |^2\,dx + \int_\Omega e^{-s} \widetilde{\beta}(s) | \nabla \partial_s( \widetilde{w}^\varepsilon - w^\varepsilon )(s) |^2\,dx \\[1ex]
  &\geq \delta_\alpha e^{-s} \norm{ \partial_s (\widetilde{w}^\varepsilon - w^\varepsilon)(s) }_H^2, \forae s \in (0,S),
\end{align*}
which implies
\begin{gather*}
    \partial_s (\widetilde{w}^\varepsilon-w^\varepsilon)(s) = 0
  \mbox{ in } H, \ae s \in (0, S).
\end{gather*}
With $\widetilde{w}^\varepsilon(0)=w^\varepsilon(0)=u_0$ in mind, one can see that
\begin{gather*}
  \widetilde{w}^\varepsilon(s)=w^\varepsilon(s)
  \quad \text{for any } s\in[0,S],
    \\
    \mbox{and therefore, } 
    \partial_s w^\varepsilon(s)
  = \partial_s \widetilde{w}^\varepsilon(s)
  = v^\varepsilon(s)
  \mbox{ in } H, \ae s \in (0, S).
\end{gather*}

Since $S>0$ is arbitrary, the assertion (ii) follows, and the proof of Lemma \ref{lem 4-3} is complete.
\end{proof}

\begin{proof}[Proof of Main Theorem 4]
    Let $\varepsilon \in [0,1)$, and let $ u^\varepsilon_\infty \in \omega(u^\varepsilon) $ be the $\omega$-limit point of the unique solution $u^\varepsilon$ to (P)$_\varepsilon$. With the time-rescaling \eqref{4.3.1} and the associated functions as in  \eqref{4.3.3},  we invoke the proof of Main Theorem 3 to take a sequence $\{ t^\varepsilon_n \} \subset (0,\infty)$, with $t^\varepsilon_n \uparrow \infty$, satisfying \eqref{4.2.4}--\eqref{4.2.7}.
    
Let us set
\begin{gather*}
  s^\varepsilon_n \coloneqq \log{ (t^\varepsilon_n + 1) }, \text{ for any } n \in \mathbb N.
\end{gather*}
In view of the relation $w^\varepsilon(s)=u^\varepsilon(e^s-1)$, \eqref{4.2.4} can be rewritten as follows:
\begin{gather}
    w^\varepsilon(s^\varepsilon_n) \to u^\varepsilon_\infty
    \begin{cases}
      & \In H,\\[1ex]
      \weaklystar & \In L^\infty(\Omega).\\[1ex]
      \weaklystar & \In BV(\Omega),
    \end{cases}
    \ \ \as n\to\infty.\label{4.3.8}\noeqref{4.3.8}
  \end{gather}
Due to the time-rescaling in \eqref{4.3.1}, the coefficient $e^{-s}\widetilde{\beta}$ exhibits exponential decay. Once \eqref{4.3.2} is verified, the argument used in the proof of Main Theorem 3, with $\beta=e^{-s}\widetilde{\beta}$, shows that every $\omega$-limit point belongs to $S^\varepsilon_\infty$. Hence,
\begin{itemize}
\item[$\sharp$5)] $\omega(u^\varepsilon) \subset S^\varepsilon_\infty$ with $ \omega(u^\varepsilon) \supset {u^\varepsilon_\infty} \ne \emptyset $.
\end{itemize}
On the other hand, the strict convexity of $G$ assumed in (a-3) in (A7) implies that the energy $\mathcal F_\varepsilon$ is strictly convex. It follows that
\begin{itemize}
\item[$\sharp$6)] $S^\varepsilon_\infty$ is a singleton $\{u^\varepsilon_\infty\}$ consisting of the unique steady-state solution $u^\varepsilon_\infty \in S^\varepsilon_\infty $, equivalently, the unique minimizer of $\mathcal F_\varepsilon$.
\end{itemize}
Accordingly, the proof of Main Theorem 4 reduces to verifying \eqref{4.3.2}.
    \medskip

    To verify \eqref{4.3.2}, we first derive the basic energy estimates for (PS)$_\varepsilon$.
The corresponding energy inequality is given by
\begin{gather}
  \frac{1}{2} \int_0^s \int_\Omega e^{-\sigma} \widetilde{\alpha}(\sigma) | \partial_s w^\varepsilon(\sigma) |^2\,dxd\sigma + \int_0^s \int_\Omega e^{-\sigma} \widetilde{\beta}(\sigma) | \nabla \partial_s w^\varepsilon(\sigma) |^2 + \mathcal F_\varepsilon(w^\varepsilon(s))\\[1ex]
  \leq\frac{1}{\delta_\alpha} \int_0^s \int_\Omega e^\sigma | \widetilde{f}(\sigma) - f_\infty |^2\,dxd\sigma + \mathcal F_\varepsilon(u_0), \text{ for any } s \in [0,\infty).\label{4.3.9}\noeqref{4.3.9}
\end{gather}
        By the change of variables $t=e^\sigma-1$, one can see from (A6) that the first term on the right-hand side is uniformly bounded. Hence, following the same argument as in Lemma \ref{lem 4-1}, we obtain that:
\begin{itemize}
  \item[$\sharp$7)] $ e^{ -\frac{s}{2} } \partial_s w^\varepsilon \in \mathscr H $;
      \vspace{-1ex}
  \item[$\sharp$8)] $ e^{ -\frac{s}{2} } \sqrt{\widetilde{\beta}} \nabla \partial_s w^\varepsilon \in [\mathscr H]^N $;
      \vspace{-1ex}
  \item[$\sharp$9)] $ w^\varepsilon \in L^\infty(0,\infty; W^{1,1}(\Omega) \cap L^\infty(\Omega)) $.
\end{itemize}
    As noted earlier, the estimate in $\sharp$7) is weighted by $e^{-s/2}$ and therefore does not
yet imply \eqref{4.3.2}. To remove this weight, we use the linearized problem
    (LPS)$_\varepsilon$ together with $\sharp$7) and $\sharp$8).

Now, we divide the proof into two cases: $\varepsilon \in (0,1)$ and $\varepsilon = 0$.
\medskip\\
    \noindent \underline{\textit{Proof for $\varepsilon \in (0,1)$.}}\ \ 
    By Lemma \ref{lem 4-3}, the solution to (LPS)$_\varepsilon$ satisfies $v^\varepsilon=\partial_s w^\varepsilon$. Testing (LPS)$_\varepsilon$
with $\partial_s w^\varepsilon$ and using the monotonicity of
$\partial\gamma_\varepsilon$ together with (A7), we obtain
\begin{gather*}
  \frac{d}{ds} \left\{ \norm{e^{ -\frac{s}{2} } \sqrt{\widetilde{\alpha}(s)} \partial_s w^\varepsilon(s) }_H^2 + \norm{e^{ -\frac{s}{2} } \sqrt{\widetilde{\beta}(s)} \nabla \partial_s w^\varepsilon(s) }_{[H]^N}^2 \right\} + C_{g^\prime} \norm{\partial_s w^\varepsilon(s)}_H^2\\[1ex]
  \leq (1 + C_\alpha + C_\beta) \left\{ \norm{e^{ -\frac{s}{2} } \sqrt{\widetilde{\alpha}(s)} \partial_s w^\varepsilon(s) }_H^2 + \norm{e^{ -\frac{s}{2} } \sqrt{\widetilde{\beta}(s)} \nabla \partial_s w^\varepsilon(s) }_{[H]^N}^2 \right\} + \frac{1}{ C_{g^\prime} } \norm{ \partial_s \widetilde{f}(s) }_H^2,\\[1ex]
 \forae s \in (0,\infty).
\end{gather*}
    Here, for any $ s \in [0, \infty) $, integrating this differential inequality over $(0,s)$ gives
\begin{gather}
  \norm{e^{ -\frac{s}{2} } \sqrt{\widetilde{\alpha}(s)} \partial_s w^\varepsilon(s) }_H^2 + \norm{e^{ -\frac{s}{2} } \sqrt{\widetilde{\beta}(s)} \nabla \partial_s w^\varepsilon(s) }_{[H]^N}^2 + C_{g^\prime} \int_0^s \norm{\partial_s w^\varepsilon (\sigma)}_H^2\,d\sigma\\[1ex]
  \leq \norm{ \sqrt{\alpha(0)} \partial_s w^\varepsilon (0) }_H^2 + \norm{ \sqrt{\beta(0)} \nabla \partial_s w^\varepsilon (0) }_{[H]^N}^2 + \frac{1}{C_{g^\prime}} \norm{\partial_s \widetilde{f}}_\mathscr H^2\\[1ex]
  + (1 + C_\alpha + C_\beta) \left\{ \norm{ e^{ -\frac{s}{2} } \sqrt{\widetilde{\alpha}} \partial_s w^\varepsilon }_\mathscr H^2 + \norm{ e^{ -\frac{s}{2} }\sqrt{\widetilde{\beta}} \nabla \partial_s w^\varepsilon }_{[\mathscr H]^N}^2 \right\}. 
    \hspace{2em}\label{4.3..10}\noeqref{4.3..10}
\end{gather}
Combining this estimate with $\sharp$7), $\sharp$8), assumption (a-4) in (A7), and the prescribed initial condition $\partial_s w^\varepsilon(0) = v^\varepsilon(0)$, we conclude that \eqref{4.3.2} holds for every $\varepsilon \in (0,1)$.

This completes the proof for $\varepsilon \in (0,1)$.
\medskip\\

\noindent \underline{\textit{Proof for $\varepsilon = 0$.}}\ \ 
In this case, the method used for $\varepsilon \in (0,1)$ does not apply, since the Euclidean norm $\gamma_0$ does not admit a second derivative, which prevents us from formulating the corresponding linearized problem.
However, \eqref{4.3.2} for $\varepsilon=0$ can be obtained by passing to the limit as $\varepsilon\to0$ in the estimates established for $\varepsilon\in(0,1)$. Indeed, from \eqref{4.3.9}, \eqref{4.3..10}, and the initial condition $\partial_s w^\varepsilon(0)=v^\varepsilon(0)$ in $H$, it follows that
\begin{align}
  &C_{g^\prime} \int_0^\infty \norm{\partial_s w^\varepsilon (\sigma)}_H^2\,d\sigma \leq \frac{1}{\delta_\alpha \land \delta_1} \bigl( \norm{ f(0) - g(u_0) }_H^2 + \mathcal L^N (\Omega) \bigr) + \frac{1}{C_{g^\prime}} \norm{\partial_s \widetilde{f}}_\mathscr H^2\\[1ex]
  &\hspace{4em}+ (1 + C_\alpha + C_\beta) \left\{ \frac{1}{\delta_\alpha} \norm{ f - f_\infty }_\mathscr H^2 + \sup_{\varepsilon \in (0,1)} \mathcal F_\varepsilon(u_0) \right\}, \text{ for any } \varepsilon \in (0,1),\hspace{3em}\label{4.3..11}\noeqref{4.3..11}
\end{align}
Here, we have used
\begin{align*}
  &\norm{\sqrt{\alpha(0)} \partial_s w^\varepsilon(0)}_H^2 + \norm{ \sqrt{\beta(0)} \nabla \partial_s w^\varepsilon(0) }_{[H]^N}^2\\[1ex]
 &\hspace{2em}= \bigl( f(0) - g(u_0), \partial_s w^\varepsilon(0) \bigr)_H - \bigl( \nabla \gamma_\varepsilon(\nabla u_0), \nabla \partial_s w^\varepsilon(0) \bigr)_{[H]^N},
\end{align*}
which implies
\begin{gather*}
  \norm{\sqrt{\alpha(0)} \partial_s w^\varepsilon(0)}_H^2 + \norm{ \sqrt{\beta(0)} \nabla \partial_s w^\varepsilon(0) }_{[H]^N}^2 \leq \frac{1}{\delta_\alpha \land \delta_1} \bigl( \norm{ f(0) - g(u_0) }_H^2 + \mathcal L^N (\Omega) \bigr),
\end{gather*}
together with
\begin{gather*}
  \norm{ e^{ -\frac{s}{2} } \sqrt{\widetilde{\alpha}} \partial_s w^\varepsilon }_\mathscr H^2 + \norm{ e^{ -\frac{s}{2} }\sqrt{\widetilde{\beta}} \nabla \partial_s w^\varepsilon }_{[\mathscr H]^N}^2 \leq \frac{1}{\delta_\alpha} \norm{ f - f_\infty }_\mathscr H^2 + \sup_{\varepsilon \in (0,1)} \mathcal F_\varepsilon(u_0), \text{ for any }\varepsilon \in (0,1).
\end{gather*}
Estimate \eqref{4.3..11} is uniform with respect to $\varepsilon \in (0,1)$, so that
\begin{itemize}
  \item[$\flat$8)] $ \{ \partial_s w^\varepsilon \}_{\varepsilon \in (0,1)} \text{ is bounded in } \mathscr H $.
\end{itemize}

In the meantime, taking into account \eqref{4.3.1} and Remark \ref{eps-dep}, one can further derive:
\begin{gather}
    (t +1) \partial_t u^\varepsilon \to (t +1) \partial_t u^0\weakly \In L^2_\loc([0,\infty);H), 
    \\
    \mbox{equivalently, }
    \partial_s w^\varepsilon  \to \partial_s w^0 \weakly \In L^2_\loc([0,\infty);H), \mbox{ as $ \varepsilon \to 0 $.}\label{4.3..13}\noeqref{4.3..13}
\end{gather}
    $\flat$8) and \eqref{4.3..13} imply that \eqref{4.3.2} holds for $\varepsilon=0$.
    \medskip

    Thus, we conclude Main Theorem 4. 
\end{proof}

\section{Appendix}
In this appendix, we prove the lemma used in Main Theorem 2.
\begin{lemma}\label{lem 5-1}
  Let $0 \leq a \in C([0,\infty))$ be an unbounded function. Then, for any $k > 0$, there exists a constant $t_k \in [k,\infty)$ such that
\begin{gather*}
  a(t_k) \leq a(t_k + 1).
\end{gather*}
\end{lemma}
\begin{proof}
  We prove this lemma by contradiction. Let us suppose that there exists a constant $k_0 > 0$ satisfying
\begin{gather}
  a(t) > a(t + 1) \geq 0, \text{ for any } t \in [k_0,\infty).\label{5.1.2}\noeqref{5.1.2}
\end{gather}
    Besides, let us take any $t \in [k_0, \infty)$, together with $\delta \in [k_0, k_0 +1)$ and $ 0 \leq m \in \mathbb Z \cup \{ 0 \}$ such that $t = \delta + m$.
  Then, \eqref{5.1.2} yields 
    \begin{gather*} 
        a(\delta) > a(\delta + 1) > a(\delta + 2) > \cdots > a(\delta + m) = a(t). 
    \end{gather*} 
    Since $a$ is continuous, it is bounded on the compact interval $[k_0,k_0+1]$. Hence, 
    \begin{gather*} 
        0 \leq a(t) \leq \max_{\delta \in [k_0,k_0+1]} a(\delta) < \infty, \text{ for any } t \in [k_0,\infty). 
    \end{gather*} 
    Moreover, the continuity of $a$ also implies that $a$ is bounded on $[0,k_0+1]$. Therefore, $a$ is bounded on $[0,\infty)$, which contradicts the assumption that $a$ is unbounded. 

    Thus, Lemma \ref{lem 5-1} holds.
\end{proof}

\bibliography{sinst41}

@book {MR0773850,
    AUTHOR = {Attouch, H.},
     TITLE = {Variational Convergence for Functions and Operators},
    SERIES = {Applicable Mathematics Series},
 PUBLISHER = {Pitman (Advanced Publishing Program), Boston, MA},
      YEAR = {1984},
     PAGES = {xiv+423},
      ISBN = {0-273-08583-2},
   MRCLASS = {49-02 (49A50)},
  MRNUMBER = {0773850},
MRREVIEWER = {Joachim Naumann},
}

@book {MR2582280,
    AUTHOR = {Barbu, Viorel},
     TITLE = {Nonlinear Differential Equations of Monotone Types in {B}anach Spaces},
    SERIES = {Springer Monographs in Mathematics},
 PUBLISHER = {Springer, New York},
      YEAR = {2010},
     PAGES = {x+272},
      ISBN = {978-1-4419-5541-8},
   MRCLASS = {34-02 (34G20 34G25 35A16 47J35 47N20)},
  MRNUMBER = {2582280},
MRREVIEWER = {Jean Mawhin},
       DOI = {10.1007/978-1-4419-5542-5},
       URL = {http://dx.doi.org/10.1007/978-1-4419-5542-5},
}

@book {MR0348562,
    AUTHOR = {Br\'ezis, Haim},
     TITLE = {Op\'erateurs Maximaux Monotones et Semi-groupes de Contractions dans les Espaces de {H}ilbert},
      NOTE = {North-Holland Mathematics Studies, No. 5. Notas de Matem\'atica (50)},
 PUBLISHER = {North-Holland Publishing Co., Amsterdam-London; American Elsevier Publishing Co., Inc., New York},
      YEAR = {1973},
     PAGES = {vi+183},
   MRCLASS = {47H05},
  MRNUMBER = {0348562},
MRREVIEWER = {Bruce Calvert},
}

@article {MR2096945,
    AUTHOR = {Giga, Yoshikazu and Kashima, Yohei and Yamazaki, Noriaki},
     TITLE = {Local solvability of a constrained gradient system of total
              variation},
   JOURNAL = {Abstr. Appl. Anal.},
  FJOURNAL = {Abstract and Applied Analysis},
      YEAR = {2004},
    NUMBER = {8},
     PAGES = {651--682},
      ISSN = {1085-3375},
   MRCLASS = {35K55 (26A45 35A15 35J20 35K90 49J45 58E20)},
  MRNUMBER = {2096945},
MRREVIEWER = {Koji Kikuchi},
       DOI = {10.1155/S1085337504311048},
       URL = {https://doi.org/10.1155/S1085337504311048},
}

@book {MR0775682,
    AUTHOR = {Giusti, Enrico},
     TITLE = {Minimal Surfaces and Functions of Bounded Variation},
    SERIES = {Monographs in Mathematics},
    VOLUME = {80},
 PUBLISHER = {Birkh\"auser Verlag, Basel},
      YEAR = {1984},
     PAGES = {xii+240},
      ISBN = {0-8176-3153-4},
   MRCLASS = {58E12 (49F10 53A10)},
  MRNUMBER = {0775682},
MRREVIEWER = {Helmut Kaul},
       DOI = {10.1007/978-1-4684-9486-0},
       URL = {http://dx.doi.org/10.1007/978-1-4684-9486-0},
}

@article {Kenmochi81,
   AUTHOR = {N. Kenmochi},
	TITLE = {Solvability of nonlinear evolution equations with time-dependent constraints and applications},
	JOURNAL = {Bull. Fac. Education, Chiba Univ. (\url{http://ci.nii.ac.jp/naid/110004715232})}, 
	VOLUME = {30}, 
	YEAR = {1981}, 
	PAGES = {1--87}, 
}

@article {MR1752970,
    AUTHOR = {Kobayashi, Ryo and Warren, James A. and Carter, W. Craig},
     TITLE = {A continuum model of grain boundaries},
   JOURNAL = {Phys. D},
  FJOURNAL = {Physica D. Nonlinear Phenomena},
    VOLUME = {140},
      YEAR = {2000},
    NUMBER = {1-2},
     PAGES = {141--150},
      ISSN = {0167-2789},
   MRCLASS = {74N05},
  MRNUMBER = {1752970},
       DOI = {10.1016/S0167-2789(00)00023-3},
       URL = {http://dx.doi.org/10.1016/S0167-2789(00)00023-3},
}

@incollection {MR1794359,
    AUTHOR = {Kobayashi, Ryo and Warren, James A. and Carter, W. Craig},
     TITLE = {Grain boundary model and singular diffusivity},
 BOOKTITLE = {Free boundary problems: theory and applications, {II}
              ({C}hiba, 1999)},
    SERIES = {GAKUTO Internat. Ser. Math. Sci. Appl.},
    VOLUME = {14},
     PAGES = {283--294},
 PUBLISHER = {Gakk\=otosho, Tokyo},
      YEAR = {2000},
   MRCLASS = {74N05 (35J60 35K55 35R35 74N20)},
  MRNUMBER = {1794359},
}

@article {MR0298508,
    AUTHOR = {Mosco, Umberto},
     TITLE = {Convergence of convex sets and of solutions of variational
              inequalities},
   JOURNAL = {Advances in Math.},
  FJOURNAL = {Advances in Mathematics},
    VOLUME = {3},
      YEAR = {1969},
     PAGES = {510--585},
      ISSN = {0001-8708},
   MRCLASS = {47B99 (46N05 52A40)},
  MRNUMBER = {0298508},
MRREVIEWER = {C. Fenske},
       DOI = {10.1016/0001-8708(69)90009-7},
       URL = {http://dx.doi.org/10.1016/0001-8708(69)90009-7},
}

@incollection {MR2232843,
    AUTHOR = {Shirakawa, Ken},
     TITLE = {Stability analysis for {A}llen-{C}ahn equations involving
              indefinite diffusion coefficients},
 BOOKTITLE = {Mathematical approach to nonlinear phenomena: modelling,
              analysis and simulations},
    SERIES = {GAKUTO Internat. Ser. Math. Sci. Appl.},
    VOLUME = {23},
     PAGES = {238--254},
 PUBLISHER = {Gakk\=otosho, Tokyo},
      YEAR = {2005},
   MRCLASS = {35K55 (34D30 34G25 35B35 82C26)},
  MRNUMBER = {2232843},
MRREVIEWER = {Enrico Valdinoci},
}

@article {MR0916688,
    AUTHOR = {Simon, Jacques},
     TITLE = {Compact sets in the space {$L^p(0,T;B)$}},
   JOURNAL = {Ann. Mat. Pura Appl. (4)},
  FJOURNAL = {Annali di Matematica Pura ed Applicata. Serie Quarta},
    VOLUME = {146},
      YEAR = {1987},
     PAGES = {65--96},
      ISSN = {0003-4622},
   MRCLASS = {46E40 (46E30)},
  MRNUMBER = {0916688},
MRREVIEWER = {James Bell Cooper},
       DOI = {10.1007/BF01762360},
       URL = {http://dx.doi.org/10.1007/BF01762360},
}

@book {MR1423808,
    AUTHOR = {Visintin, Augusto},
     TITLE = {Models of Phase Transitions},
    SERIES = {Progress in Nonlinear Differential Equations and their
              Applications},
    VOLUME = {28},
 PUBLISHER = {Birkh\"auser Boston, Inc., Boston, MA},
      YEAR = {1996},
     PAGES = {x+322},
      ISBN = {0-8176-3768-0},
   MRCLASS = {80A22 (35K05 49J45 82C26)},
  MRNUMBER = {1423808},
MRREVIEWER = {Peter G. Danilaev},
       DOI = {10.1007/978-1-4612-4078-5},
       URL = {http://dx.doi.org/10.1007/978-1-4612-4078-5},
}

@article {MR3661429,
    AUTHOR = {Colli, Pierluigi and Gilardi, Gianni and Nakayashiki, Ryota
              and Shirakawa, Ken},
     TITLE = {A class of quasi-linear {A}llen--{C}ahn type equations with
              dynamic boundary conditions},
   JOURNAL = {Nonlinear Anal.},
  FJOURNAL = {Nonlinear Analysis. Theory, Methods \& Applications. An
              International Multidisciplinary Journal},
    VOLUME = {158},
      YEAR = {2017},
     PAGES = {32--59},
      ISSN = {0362-546X},
   MRCLASS = {35K55 (35K59 82C26)},
  MRNUMBER = {3661429},
       DOI = {10.1016/j.na.2017.03.020},
       URL = {http://dx.doi.org/10.1016/j.na.2017.03.020},
}

@book {MR2033382,
    AUTHOR = {Andreu-Vaillo, Fuensanta and Caselles, Vicent and Maz\'on, Jos\'e
              M.},
     TITLE = {Parabolic quasilinear equations minimizing linear growth
              functionals},
    SERIES = {Progress in Mathematics},
    VOLUME = {223},
 PUBLISHER = {Birkh\"auser Verlag, Basel},
      YEAR = {2004},
     PAGES = {xiv+340},
      ISBN = {3-7643-6619-2},
   MRCLASS = {35-02 (35A15 35K55 47H20 49J10 49K10 94A08)},
  MRNUMBER = {2033382},
MRREVIEWER = {Gary M. Lieberman},
       DOI = {10.1007/978-3-0348-7928-6},
       URL = {http://dx.doi.org/10.1007/978-3-0348-7928-6},
}

@article {MR2746654,
    AUTHOR = {Giga, Mi-Ho and Giga, Yoshikazu},
     TITLE = {Very singular diffusion equations: second and fourth order
              problems},
   JOURNAL = {Jpn. J. Ind. Appl. Math.},
  FJOURNAL = {Japan Journal of Industrial and Applied Mathematics},
    VOLUME = {27},
      YEAR = {2010},
    NUMBER = {3},
     PAGES = {323-345},
      ISSN = {0916-7005},
   MRCLASS = {35K59 (35K25 62H35 74N05)},
  MRNUMBER = {2746654},
       DOI = {10.1007/s13160-010-0020-y},
       URL = {http://dx.doi.org/10.1007/s13160-010-0020-y},
}

@article {MR1712447,
    AUTHOR = {Kobayashi, R. and Giga, Y.},
     TITLE = {Equations with singular diffusivity},
   JOURNAL = {J. Statist. Phys.},
  FJOURNAL = {Journal of Statistical Physics},
    VOLUME = {95},
      YEAR = {1999},
    NUMBER = {5-6},
     PAGES = {1187--1220},
      ISSN = {0022-4715},
   MRCLASS = {82C31 (35K65 74N05 82C24)},
  MRNUMBER = {1712447},
       DOI = {10.1023/A:1004570921372},
       URL = {http://dx.doi.org/10.1023/A:1004570921372},
}

@book {MR932730,
    AUTHOR = {Vrabie, I. I.},
     TITLE = {Compactness methods for nonlinear evolutions},
    SERIES = {Pitman Monographs and Surveys in Pure and Applied Mathematics},
    VOLUME = {32},
      NOTE = {With a foreword by A. Pazy},
 PUBLISHER = {Longman Scientific \& Technical, Harlow; John Wiley \& Sons,
              Inc., New York},
      YEAR = {1987},
     PAGES = {xviii+325},
      ISBN = {0-582-00319-9},
   MRCLASS = {47H20 (34-02 35-02)},
  MRNUMBER = {932730},
}

@article{MR1847840,
  title={Stability for a parabolic variational inequality associated with total variation functional},
  author={Kenmochi, Nobuyuki and Shirakawa, Ken},
  journal={Funkcial. Ekvac.},
  volume={44},
  number={1},
  pages={119--137},
  year={2001},
  publisher={}
}

@article{MR1851862,
  title={A variational inequality for total variation functional with constraint},
  author={Kenmochi, Nobuyuki and Shirakawa, Ken},
  journal={Nonlinear Anal.},
  volume={46},
  number={3},
  pages={435--455},
  year={2001},
  publisher={}
}

@article {MR2028858,
    AUTHOR = {Vese, Luminita A. and Osher, Stanley J.},
     TITLE = {Modeling textures with total variation minimization and
              oscillating patterns in image processing},
      NOTE = {Special issue in honor of the sixtieth birthday of Stanley
              Osher},
   JOURNAL = {J. Sci. Comput.},
  FJOURNAL = {Journal of Scientific Computing},
    VOLUME = {19},
      YEAR = {2003},
    NUMBER = {1-3},
     PAGES = {553--572},
      ISSN = {0885-7474,1573-7691},
   MRCLASS = {49J10 (65M06 68T45 94A08)},
  MRNUMBER = {2028858},
MRREVIEWER = {Ilya\ S.\ Molchanov},
       DOI = {10.1023/A:1025384832106},
       URL = {https://doi.org/10.1023/A:1025384832106},
}

@article {MR2170510,
    AUTHOR = {Tsai, Yen-Hsi Richard and Osher, Stanley},
     TITLE = {Total variation and level set methods in image science},
   JOURNAL = {Acta Numer.},
  FJOURNAL = {Acta Numerica},
    VOLUME = {14},
      YEAR = {2005},
     PAGES = {509--573},
      ISSN = {0962-4929,1474-0508},
   MRCLASS = {94A08 (35A15 35K55 68T45 68U10)},
  MRNUMBER = {2170510},
MRREVIEWER = {Ram\'{o}n\ M.\ Rodr\'{\i}guez-Dagnino},
       DOI = {10.1017/S0962492904000273},
       URL = {https://doi.org/10.1017/S0962492904000273},
}

@article {MR4352617,
    AUTHOR = {Moll, Salvador and Shirakawa, Ken and Watanabe, Hiroshi},
     TITLE = {Kobayashi-{W}arren-{C}arter type systems with nonhomogeneous
              {D}irichlet boundary data for crystalline orientation},
   JOURNAL = {Nonlinear Anal.},
  FJOURNAL = {Nonlinear Analysis. Theory, Methods \& Applications. An
              International Multidisciplinary Journal},
    VOLUME = {217},
      YEAR = {2022},
     PAGES = {Paper No. 112722, 44},
      ISSN = {0362-546X,1873-5215},
   MRCLASS = {35K87 (35K67 35R06)},
  MRNUMBER = {4352617},
       DOI = {10.1016/j.na.2021.112722},
       URL = {https://doi.org/10.1016/j.na.2021.112722},
}

@article{aiki2023class,
  title={A Class of Initial-Boundary Value Problems Governed by Pseudo-Parabolic Weighted Total Variation Flows},
  author={Aiki, Toyohiko and Mizuno, Daiki and Shirakawa, Ken},
  journal={Advances in Mathematical Sciences and Applications},
  FJOURNAL = {Advances in Mathematical Sciences and Applications},
  VOLUME = {32},
  YEAR = {2023},
  PAGES = {311--341},
  ISSN = {1343-4373},
  URL = {https://mcm-www.jwu.ac.jp/~aikit/AMSA/pdf/abstract/2023/Top_2023_015.pdf},
}

@incollection {MR2571495,
    AUTHOR = {Shirakawa, Ken},
     TITLE = {Continuous dependence for solution classes of
              {E}uler-{L}agrange equations generated by linear growth
              energies},
 BOOKTITLE = {Nonlocal and abstract parabolic equations and their
              applications},
    SERIES = {Banach Center Publ.},
    VOLUME = {86},
     PAGES = {287--302},
 PUBLISHER = {Polish Acad. Sci. Inst. Math., Warsaw},
      YEAR = {2009},
      ISBN = {978-83-86806-05-8},
   MRCLASS = {35J60 (35J20 49K40)},
  MRNUMBER = {2571495},
MRREVIEWER = {Jianqing\ Chen},
       DOI = {10.4064/bc86-0-18},
       URL = {https://doi.org/10.4064/bc86-0-18},
}

@article {MR2368971,
    AUTHOR = {Caselles, Vicent and Chambolle, Antonin and Novaga, Matteo},
     TITLE = {The discontinuity set of solutions of the {TV} denoising
              problem and some extensions},
   JOURNAL = {Multiscale Model. Simul.},
  FJOURNAL = {Multiscale Modeling \& Simulation. A SIAM Interdisciplinary
              Journal},
    VOLUME = {6},
      YEAR = {2007},
    NUMBER = {3},
     PAGES = {879--894},
      ISSN = {1540-3459,1540-3467},
   MRCLASS = {94A08 (35J20 35J70 65J20 68U10)},
  MRNUMBER = {2368971},
MRREVIEWER = {Ryuichi\ Ashino},
       DOI = {10.1137/070683003},
       URL = {https://doi.org/10.1137/070683003},
}

@article {MR1883415,
    AUTHOR = {Andreu, F. and Caselles, V. and D\'iaz, J. I. and Maz\'on, J.
              M.},
     TITLE = {Some qualitative properties for the total variation flow},
   JOURNAL = {J. Funct. Anal.},
  FJOURNAL = {Journal of Functional Analysis},
    VOLUME = {188},
      YEAR = {2002},
    NUMBER = {2},
     PAGES = {516--547},
      ISSN = {0022-1236,1096-0783},
   MRCLASS = {35K55 (35K20)},
  MRNUMBER = {1883415},
MRREVIEWER = {Qing\ Fang},
       DOI = {10.1006/jfan.2001.3829},
       URL = {https://doi.org/10.1006/jfan.2001.3829},
}

@article {MR2139202,
    AUTHOR = {Bellettini, Giovanni and Caselles, Vicent and Novaga, Matteo},
     TITLE = {Explicit solutions of the eigenvalue problem {$-{\rm
              div}\left(\frac{Du}{|Du|}\right)=u$} in {$\bold R^2$}},
   JOURNAL = {SIAM J. Math. Anal.},
  FJOURNAL = {SIAM Journal on Mathematical Analysis},
    VOLUME = {36},
      YEAR = {2005},
    NUMBER = {4},
     PAGES = {1095--1129},
      ISSN = {0036-1410,1095-7154},
   MRCLASS = {35P30 (35C05 35J60 35K55 94A08)},
  MRNUMBER = {2139202},
MRREVIEWER = {Marino\ Belloni},
       DOI = {10.1137/S0036141003430007},
       URL = {https://doi.org/10.1137/S0036141003430007},
}

@article {MR2178065,
    AUTHOR = {Alter, F. and Caselles, V. and Chambolle, A.},
     TITLE = {A characterization of convex calibrable sets in {$\Bbb R^N$}},
   JOURNAL = {Math. Ann.},
  FJOURNAL = {Mathematische Annalen},
    VOLUME = {332},
      YEAR = {2005},
    NUMBER = {2},
     PAGES = {329--366},
      ISSN = {0025-5831,1432-1807},
   MRCLASS = {35J70 (35K65 49J40 52A20)},
  MRNUMBER = {2178065},
MRREVIEWER = {Jes\'us\ Hern\'andez},
       DOI = {10.1007/s00208-004-0628-9},
       URL = {https://doi.org/10.1007/s00208-004-0628-9},
}

@incollection {MR3363401,
    AUTHOR = {Rudin, Leonid I. and Osher, Stanley and Fatemi, Emad},
     TITLE = {Nonlinear total variation based noise removal algorithms},
      NOTE = {Experimental mathematics: computational issues in nonlinear
              science (Los Alamos, NM, 1991)},
   JOURNAL = {Phys. D},
  FJOURNAL = {Physica D. Nonlinear Phenomena},
    VOLUME = {60},
      YEAR = {1992},
    NUMBER = {1-4},
     PAGES = {259--268},
      ISSN = {0167-2789,1872-8022},
   MRCLASS = {94A08 (65D18 68U10)},
  MRNUMBER = {3363401},
       DOI = {10.1016/0167-2789(92)90242-F},
       URL = {https://doi.org/10.1016/0167-2789(92)90242-F},
}

@inproceedings{rudin1994total,
  title={Total variation based image restoration with free local constraints},
  author={Rudin, Leonid I and Osher, Stanley},
  booktitle={Proceedings of 1st international conference on image processing},
  volume={1},
  pages={31--35},
  year={1994},
  organization={IEEE}
}

@incollection {MR2028838,
    AUTHOR = {Bertalmio, M. and Caselles, V. and Roug\'e, B. and Sol\'e, A.},
     TITLE = {T{V} based image restoration with local constraints},
      NOTE = {Special issue in honor of the sixtieth birthday of Stanley
              Osher},
   JOURNAL = {J. Sci. Comput.},
  FJOURNAL = {Journal of Scientific Computing},
    VOLUME = {19},
      YEAR = {2003},
    NUMBER = {1-3},
     PAGES = {95--122},
      ISSN = {0885-7474,1573-7691},
   MRCLASS = {94A08 (65J20 68U10 90C90)},
  MRNUMBER = {2028838},
MRREVIEWER = {H.\ P.\ Dikshit},
       DOI = {10.1023/A:1025391506181},
       URL = {https://doi.org/10.1023/A:1025391506181},
}

@book {MR348562,
    AUTHOR = {Br\'ezis, H.},
     TITLE = {Op\'erateurs maximaux monotones et semi-groupes de
              contractions dans les espaces de {H}ilbert},
    SERIES = {North-Holland Mathematics Studies},
    VOLUME = {No. 5},
      NOTE = {Notas de Matem\'atica, No. 50. [Mathematical Notes]},
 PUBLISHER = {North-Holland Publishing Co., Amsterdam-London; American
              Elsevier Publishing Co., Inc., New York},
      YEAR = {1973},
     PAGES = {vi+183},
   MRCLASS = {47H05},
  MRNUMBER = {348562},
MRREVIEWER = {Bruce\ Calvert},
}

\end{document}